\documentclass{amsart}
\usepackage[utf8]{inputenc}
\usepackage[T1]{fontenc}
\usepackage{amsmath}
\usepackage{amsfonts}
\usepackage{amsthm}
\usepackage{amssymb}
\usepackage{enumerate}
\usepackage{geometry}
\usepackage[colorlinks=true, allcolors=blue]{hyperref}
\usepackage{bm}
\usepackage{mathrsfs}

\usepackage{cite}

\numberwithin{equation}{section}

\newtheorem{thm}{Theorem}[section]

\newtheorem{rmq}[thm]{Remark}
\newtheorem{lemm}[thm]{Lemma}

\newtheorem{prop}[thm]{Proposition}
\newtheorem{corol}[thm]{Corollary}

\newcommand{\F}{F}

\newcommand{\R}{\mathbb{R}}
\newcommand{\E}{\mathbb{E}}

\newcommand{\argmin}{\mathrm{argmin}\,}

\def\XX{{\mathbb X}}
\def\YY{{\mathbb Y}}

\def\PP{{\mathbb P}}
\def\ZZ{{\mathbb Z}}
\def\EE{{\mathbb E}} 
\def\M{\mathcal{M}}

\begin{document}

\title{{\it {k-}}means with learned metrics}
\author{Pablo Groisman}\address{Departamento de Matemática, Facultad de Ciencas Exactas y Naturales, Universidad de Buenos Aires, IMAS-CONICET and NYU-ECNU Institute of Mathematical Sciences at NYU Shanghai}
\email{pgroisma@dm.uba.ar}
\author{Matthieu Jonckheere}\address{LAAS, CNRS, 7 Avenue Colonel Roche, 31400 Toulouse, France}
\email{matthieu.jonckheere@laas.fr}
\author{Jordan Serres}\address{LPSM, Sorbonne Université, 4 place Jussieu 75005 Paris, France}
\email{jordan.serres@sorbonne-universite.fr}
\author{Mariela Sued}\address{Departamento de Matemática y Ciencias, Universidad de San Andrés, CONICET}\email{marielasued@gmail.com}

\begin{abstract}
We study the Fr\'echet {\it k-}means of a metric measure space when both the measure and the distance are unknown and have to be estimated. We prove a general result that states that the {\it k-}means are continuous with respect to the measured Gromov-Hausdorff topology. In this situation, we also prove a stability result for the Voronoi clusters they determine. We do not assume uniqueness of the set of {\it k-}means, but when it is unique, the results are stronger.
{This framework provides a unified approach to proving consistency for a wide range of metric learning procedures.
As concrete applications, we obtain new consistency results for several important estimators that were previously unestablished, even when $k=1$. These include {\it k-}means based on: (i)  Isomap and Fermat geodesic distances on manifolds, (ii) difussion distances, (iii) Wasserstein distances computed with respect to learned ground metrics. Finally, we consider applications beyond the statistical inference paradigm like (iv) first passage percolation and (v) discrete approximations of length spaces.}

\end{abstract}

\maketitle

\section{Introduction}
In many learning tasks such as clustering or dimensionality reduction, 
it is often observed that the data, while embedded in a high-dimensional ambient space, 
actually concentrate near a low-dimensional manifold $\M$ of much smaller dimension \cite{singer2006graph, Isomap, groisman2022nonhomogeneous, hein2005graphs, xu2025manifold, 
Bhattacharya, Bhattacharya2, Fefferman, JaffeKmeans}. However, in practice, the 
geometric structure of this underlying manifold is not known in advance and has to be inferred from the data. 
Learning the manifold or the underlying density can be formidable challenges in high dimensions, but estimating meaningful distances that capture both geometry and density can be more tractable and sufficient for tasks like clustering.
This observation has motivated extensive research on distance learning methods. 
Diffusion Maps \cite{coifman2006diffusion} employ a diffusion 
process on the manifold and graph Laplacians to estimate diffusion distances, which capture 
both geometric and density information through random walks on a similarity graph. Hessian 
Eigenmaps \cite{DonohoGrimes} provide a method for recovering the underlying parametrization 
of scattered data lying on a manifold $\M$ embedded in high-dimensional Euclidean space.
 Isomap \cite{Isomap}, designed initially for dimensionality reduction, estimates geodesic 
distances on the manifold. Fermat distances \cite{groisman2022nonhomogeneous} have a similar 
spirit but allow for metrics that depend not only on the geometry of the manifold but also 
on the underlying density. 
Another important instance where distances must be estimated arises in the Wasserstein space. 
In this setting, we typically do not have access to the true Wasserstein distances between 
probability measures but only to empirical estimates obtained from samples of the involved 
measures.

Despite the widespread use of these learned metrics 
in clustering and dimensionality reduction algorithms, the consistency of estimators constructed with them as an input has received limited attention.

The $k$-means clustering method provides a criterion for partitioning a cloud of points $\{X_1, \dots, X_n \}$ into $k$ groups:  first choose cluster centers that minimize the sum of the squared distance of each point to its closest center and then assign points to each nearest cluster center.
It is very fruitful to regard this procedure as a plug-in estimator of its corresponding population parameters (centers and clusters).
From this perspective, Pollard, in his seminal paper \cite{Pollard}, established the consistency of the empirical {\it k-}means for distributions on $\mathbb R^d$. Later, \cite{lember2003minimizing} extended the consistency for Borel measures defined on separable metric spaces. The case $k=1$ was considered previously by Ziezold in \cite{Ziezold}. Recently, \cite{EvansJaffe2024} introduced \textit{relaxed} notions of {\it k-}means and proved consistency for their empirical counterpart for distributions defined on metric spaces that have the Heine-Borel property.
In \cite{Cuevas} the authors prove the consistency of the empirical {\it k-}means, adapted to the setting of non-uniqueness, for probability measures defined on separable Banach spaces. They also provide necessary and sufficient probabilistic conditions for the uniqueness of the {\it k-}means set of a probability distribution and determine the asymptotic distribution of the within cluster sum of squares. As a consequence, they derive a test for uniqueness of the set of {\it k-}means.

In all these references it is assumed that the metric is known in advance.

We address here the problem of stability of $k$-means under perturbations of the entire metric measure space.
This general setting arises naturally in modern data-driven applications where the distance function is not fixed but rather evolves as a sequence that can be data-dependent (estimated through distance learning procedures) but also in randomly generated models (like first passage percolation) and discretization or approximation schemes in length spaces that can be deterministic or random.

We establish a general convergence result for both the centers and clusters obtained via the
$k$-means procedure under such varying metrics. 

Our framework considers a sequence of metric measure spaces 
$\mathcal X_n=(\XX_n,d_n,\mu_n)$ 
and proves convergence of centers and clusters
in a suitable sense, provided that the spaces $\mathcal X_n$ 
converge to a limiting space in the measured Gromov-Hausdorff (mGH) sense, see Theorem \ref{thm:convergence} and Corollary \ref{cor:consistenceclusteringundermGH}.

Note that our results hold for every $k \geq 1$ and, up to our knowledge, they are new even 
for $k=1$. We then consider many applications of our main theorem to obtain the consistency of the empirical $k$-means and their clusters for different procedures that compute meaningful distances between the data points. This is the case of the Isomap geodesic distance estimator, Fermat distances, difussion distances arising from Laplacian eigenmaps, diffusion maps and Wasserstein barycenters computed with respect to learned ground metrics. We are not aware of any consistency results for any of them, even for $k=1$. We also obtain new results concerning barycenters and $k$-barycenters in first passage percolation and for discrete approximations of length spaces.

Although our results allow for non-uniqueness of the $k$-means, the uniqueness problem is 
important for at least two reasons. On the one hand, when the set of $k$-means is unique, our 
convergence result is stronger (see equation~\eqref{conv.uniqueness}).
On the other hand, as discussed in \cite{Cuevas}, uniqueness, which is related to stability, can be used as a proxy for a good choice of $k$. In view of this, in Subsection \ref{NPC} we discuss a criterion for the uniqueness of the 1-mean under the Fermat distances. In the rest of the manuscript, we use {\it k-}means, Fréchet {\it k-}means, {\it k-}barycenters and {\it k-}medoids interchangeably. 

The remainder of the paper is organized as follows. Section~\ref{sec:mainresults} introduces 
the necessary background on measured Gromov-Hausdorff convergence and presents our main consistency 
theorems for both centers and clusters. 
Section~\ref{sec:applications} contains detailed applications to the various metric learning 
scenarios mentioned above and develops the uniqueness criterion for Fermat distances.

\section{Main results and proofs}\label{sec:mainresults}

The main result of this article is the consistency of the {\it k-}means estimator (centers and clusters) even if the notion of distance between data points is learned from the data itself rather than a given one. We remark that our statements, as well as their proofs are not data based but general abstract results on the stability of {\it k-}means under perturbations of the whole metric measure space. In section \ref{sec:applications} we apply these results to specific tasks that involve learning from data.

We first introduce the required notions of convergence and main properties of the functionals involved in the definition of the $k$-means. Then we prove the convergence of the centers and the clusters under suitable mGH convergence.

\subsection{Continuity of {\it k-}means with respect to the mGH topology.} \label{sec:continuitykmean}
We start by reviewing the measured Gromov-Hausdorff convergence (mGH). In order to do that, we first recall the notion of $\varepsilon$-isometry.
A map $h\colon (\XX,d_\XX) \to (\YY,d_\YY)$ between metric spaces is called an $\varepsilon$-isometry if
\begin{itemize}
 \item For all $x,x' \in \XX$ we have $|d_\XX(x,x') - d_\YY(h(x),h(x'))| < \varepsilon$ and
 \item For all $y\in \YY$ we have $d(y,h(\XX))< \varepsilon$.
\end{itemize}

A sequence of compact normalized metric measure spaces $\mathcal X_n=(\XX_n , d_n , \mu_n )$ converges in measured Gromov–Hausdorff sense (mGH) to a compact normalized metric measure space $\mathcal X=(\XX , d , \mu )$ if and only if there exist numbers $\varepsilon_n \to 0$ and $\varepsilon_n$-isometries $h_n\colon \XX_n\to\XX$ such that $(h_n)_{\#}(\mu_n) \to \mu$ weakly on $\XX$ when $n\to\infty$ (\!\!\cite{Fukaya,sturm2006geometry}). We denote this convergence by $\mathcal X_n \overset{mGH}{\to} \mathcal X$.

Note that, the existence of the $\varepsilon_n$-isometries $h_n: \XX_n\to \XX$ is equivalent to the fact that $\{(\XX_n,d_n)\}$ converges to $(\XX,d)$ in the Gromov-Hausdorff distance, see \cite[Chapter 7]{Burago}.\\

Given a compact metric space $(\mathbb Y, d_{\mathbb Y})$, let $\mathcal P(\YY)$ denote the set of Borel probability measures on $\mathbb Y$ equipped with the weak convergence topology. For a  positive integer $k\geq 1$, let $C_k(\YY)$ denote the set of all subsets of $\YY$ of cardinality at most $k$. For $p\ge 1$ we define the functional
\begin{equation}
    \label{Phi}
\Phi_\YY\colon C_k(\YY)\times \mathcal P(\YY) \to \R,
\end{equation}
by
\[
\Phi_\YY(S,\mu) = \int_\YY d_\YY^p(y,S) \, \mu(dy) = \int_\YY \min_{x\in S} d_\YY^p(y,x) \, \mu(dy).
\]

For $A$ and $B$ in $C_k(\mathbb Y)$ the Hausdorff distance between them is given by 
$$
d_H(A,B)=\max\left\{\max_{a\in A}d_{\mathbb Y}(a,B),\, \max_{b\in B}d_{\mathbb Y}(b,A)\right\},
$$
where $d(x,C)=\min_{c\in C} d(x,c)$, for any $C\in C_k(\YY)$. 
Since $\YY$ is compact, the space $(C_k(\mathbb Y), d_H)$ is compact as well. Lemma 4 in \cite{JaffeKmeans} establishes that  $\Phi_\YY$  is continuous with respect to the Hausdorff topology in the first variable and the Wasserstein topology in the second one.
Thus, for each fixed $\mu \in \mathcal  P(\YY)$, $\Phi_\YY(\cdot, \mu)\colon C_k(\YY) \to \mathbb R$  is a continuous map on the compact space $C_k(\YY)$ and therefore it attains its minimum. The elements of $C_k(\YY)$ minimizing $\Phi_\YY(\cdot, \mu)$ are called the {\it k-}means of $\mu$. Unless strictly necessary, we do not emphasize the dependence on $p$. As expected, the most common choice is $p=2$.

Following \cite{Cuevas}, for a normalized metric measure space $\mathcal Y=(\YY, d_\YY, \mu)$ we define
\begin{equation}
\label{k_means}    
\mathcal S_\mathcal Y(k):=\{S\in C_k(\YY): \Phi_\YY(S,\mu)\leq \Phi_\YY(A,\mu)\; \hbox{for all $A\in C_k(\YY)$}\}.
\end{equation}
Observe that opposite to \cite{Cuevas} we use the whole measure metric space $\mathcal Y$ as a subscript instead of just the measure to emphasize that this functional depends also on the metric $d_\YY$ and the space $\YY$. We can now present the main result of this work.

\begin{thm}
\label{thm:convergence}
Let $\mathcal X_n = (\XX_n,d_n,\mu_n)$ and $\mathcal X=(\XX, d, \mu)$ be compact metric measure spaces such that $\mathcal X_n \overset{mGH}{\to} \mathcal X$.
Call $h^n\colon \XX_n \to \XX$ any sequence of $\varepsilon_n$-isometries involved in the mGH convergence of $(\mathcal X_n)$.  Then, for any $p\ge 1$ we have,
\begin{equation}
\label{eq:convergence}
\lim_{n\to\infty}\max_{S_n \in \mathcal S_{\mathcal X_n}(k)} \min_{S\in {\mathcal S}_{\mathcal X}(k)} d_H(h_n(S_n),S) = 0.
\end{equation}
\end{thm}
It is worth mentioning that the well-definiteness of both the $\max$ and $\min$ in \eqref{eq:convergence} can be established by usual continuity and compactness arguments.

This theorem is in the same spirit as \cite[Theorem~2]{JaffeKmeans}, with the
important difference that here the distance function is unknown and is allowed to be estimated from the data. It should be interpreted as a \emph{no–false-positive guarantee}: for every $\varepsilon>0$, there exists $N\in\mathbb{N}$ such that for
all $n\geq N$, and for every set of $k$-barycenters
$S_n \in \mathcal S_{\mathcal X_n}(k) $ of $(\XX_n, d_n, \mu_n)$, after taking its image by the $\varepsilon_n$-isometry, there exists a
corresponding set of $k$-barycenters ${S}\in \mathcal S_{\mathcal X}(k)$ of $(\XX, d, \mu)$
that is $\varepsilon$-close in the Hausdorff distance, i.e.
\[
d_H(h_n(S_n), S ) \le \varepsilon.
\]
Before turning to the proof of Theorem~\ref{thm:convergence}, we state an immediate
corollary, which recovers Jaffe’s results in the case of a fixed known distance
\cite{JaffeKmeans} and will play a crucial role when applied to Wasserstein
$k$-barycenters (see Section~\ref{sec:Wasserstein}).

\begin{corol}\label{cor:embeddedmGH}
Assume that $(\YY, d_\YY)$ is a compact metric space, and that $(\mu_n)_{n\ge 1}$
is a sequence of probability measures on $\YY$ converging weakly toward a probability
measure $\mu$ on $\YY$. Let $k\ge 1$, and denote by $\mathcal{S}_{\mathcal Y_n}(k)$ the set of
$k$-means for the metric measure space $\mathcal Y_n=(\YY, d_\YY, \mu_n)$, and by $\mathcal{S}_{\mathcal Y}(k)$
the set of $k$-means for $\mathcal Y=(\YY, d_\YY, \mu)$. Then,
\[
\lim_{n\to\infty}\max_{S_n \in \mathcal S_{\mathcal Y_n}(k)} \min_{S\in \mathcal S_{\mathcal Y}(k)} d_H(S_n,S) = 0.
\]
\end{corol}

\begin{proof}
The proof follows immediately from Theorem~\ref{thm:convergence}, once we verify that
the sequence of metric measure spaces $(\YY, d_\YY, \mu_n)$ converges to
$(\YY, d_\YY, \mu)$ in the measured Gromov-Hausdorff topology.
\end{proof}

\begin{rmq}\label{rmk:trueconvergence}
One could ask whether the stronger convergence
\[
\lim_{n\to\infty} \hat d_H\big(h_n(\mathcal S_{\mathcal X_n}(k)),\, \mathcal S_{\mathcal X}(k)\big)=0
\]
holds under the assumptions of Theorem \ref{thm:convergence}, with $\hat d_H$ the Hausdorff distance in $C(C_k(\XX))$, the set of compact subsets of $C_k(\XX)$. This stronger, two-sided Hausdorff convergence typically requires uniqueness of the population $k$-means. In the absence of uniqueness, only weaker notions of convergence 
(Kuratowski outer limit, one-sided Hausdorff, etc.) can be guaranteed, see e.g., \cite{Schotz2020,EvansJaffe2024,Carcamo2024} for discussion and examples. The theorem above says that the only way in which this convergence can fail is because $\mathcal S_{\mathcal X_n}(k)$ is too small compared to $\mathcal S_{\mathcal X}(k)$. When uniqueness holds ($\mathcal S_{\mathcal X}(k)$ is a singleton), it can not be the case and the above theorem can be re-written as
\begin{equation}
\label{conv.uniqueness}
\lim_{n\to\infty}\max_{S_n \in \mathcal S_{\mathcal X_n}(k)} d_H(h_n(S_n),\mathcal S_{\mathcal X}(k)) = 0.   
\end{equation}

\end{rmq}

\begin{rmq}\label{rmk:SturmDdistance} A very natural question following from Theorem \ref{thm:convergence} is to know whether $k$-means are continuous with respect to Sturm's $\mathbf D-$distance.
We would like to point out that convergence of metric measure spaces in this topology is equivalent to the mGH convergence, in the case of compact metric measure spaces with full supports and uniform bounds for the doubling constants and the diameters, see \cite[Lemma 3.18]{sturm2006geometry}. Those assumptions are very natural as long as we stay in the compact case. Only in the non-compact case, Sturm's $\mathbf D-$distance give a stricty weaker topology that mGH.
    
\end{rmq}

\begin{rmq}
Consider the following example. Let $\mathcal X_n=(\XX_n,d_n,\mu_n)$ be a space with three points $\XX_n=\{x_n,y_n,z_n\}\subset \R$. Set $x_n=0$, $y_n=n$ and $z_n=n^2$,   $d_n$ to be the Euclidean distance and $\mu_n(\{z_n\})=1/n=1-\mu_n(\{x_n\})$ (so that $\mu_n(\{y_n\})=0$). We have $\mathcal X_n \to \mathcal X$, the trivial (normalized) singleton space in Gromov-weak sense, but it does not converge in mGH sense. Take $k=1$. The Fréchet mean $\mathcal S_{\mathcal X_n}(1)=n$ for all $n$, while $\mathcal S_{\mathcal X}(1)=0$. One could ask if there is an $\varepsilon_n$-isometry such that $h_n(\mathcal S_{\mathcal X_n}(1))\to  \mathcal S_{\mathcal X}(1)$, but this is not possible since $\XX_n$ has three points with diverging distances and $\mathcal X$ is the singleton normalized metric measure space.
This example shows that Gromov-weak convergence is not enough to have convergence of the means (even up to $\epsilon_n-$isometries).
\end{rmq}

\begin{proof}[Proof of Theorem \ref{thm:convergence}.]
The proof consists in showing that each subsequence of $h_n(S_n)$, with $S_n \in \mathcal S_{\mathcal X_n}(k)$, has a further convergent subsequence whose limit point belongs to $\mathcal S_{\mathcal  X}(k)$.
For that purpose, note that each element of the subsequence of $h_n(S_n)$ belongs to the {compact space $C_k(\mathbb X)$.} Therefore, there exists a further convergent subsequence to some $S\in C_k(\mathbb X)$. We need to prove that $S$ belongs to $\mathcal S_{\mathcal  X}(k)$. To avoid annoying subscripts, assume that $d_H(h_n(S_n),S)\to 0$. The goal now is to show that $\Phi_{\XX}(S,\mu) \le \Phi_{\XX}(S',\mu)$ for all $S' \in C_k(\XX)$.

In the following we denote $\mu^h_n={(h_{n})}_\#(\mu_n)$.  
Observe that, since $(\XX,d)$ is compact, the weak convergence of $\mu^h_n \to  \mu$, guaranteed by the measured Gromov-Hausdorff convergence assumption, implies the convergence in the Wasserstein topology. From the mentioned  continuity of $\Phi_\XX$ (see \cite[Lemma 4]{JaffeKmeans}), we get
\begin{align*}
    \Phi_{\XX}\left(S,\mu\right) & = \lim_{n\to \infty} \Phi_{\XX}\left(h_n(S_n), \mu_n^h \right).
\end{align*}
By the mGH convergence, given $S^\prime\in C_k(\XX)$ there exists  $S_n^\prime\in C_k(\XX_n)$ such that $d_H(h_n(S_n^\prime), S^\prime)\to 0$. Invoking again the continuity of $\Phi_\XX$, we also get that
\begin{align*}
    \Phi_{\XX}\left(S^\prime,\mu\right) & = \lim_{n\to \infty} \Phi_{\XX}\left(h_n(S_n^\prime), \mu_n^h\right).
\end{align*}
The mGH convergence of $\mathcal X_n$ to $\mathcal X$ guarantees the existence of a constant $C$ such that $d_n(x,y)\leq C$  for all $x, y \in \XX_n$, for all $n$ and $d(x,y)\leq C$ for all $x, y \in \mathbb X$. This condition, combined with the $\epsilon_n$-isometry property, and a Taylor expansion of order one for  $f(x)=x^p$ implies that
\begin{equation}
    \label{quase-triangular}
\vert  d^p(h_n(x),h_n(y)) -d_n^p(x,y)\vert \leq\tilde C \varepsilon_n\;,\forall x,y\in \XX_n \;.
\end{equation}
Thus, for any $A\in C_k(\XX_n)$, we get that 
\begin{equation}
  \label{extending_iso}
 \vert \Phi_\XX(h_n(A), \mu_n^h) -\Phi_{\XX_n}(A,\mu_n)\vert \leq \tilde C\varepsilon_n.
\end{equation}

Invoking \eqref{extending_iso} twice with $A=S_n$ and $A=S_n^\prime$ respectively and recalling that $\Phi_{\XX_n}\left(S_n, \mu_n^h \right)\leq \Phi_{\XX_n}\left(S_n^\prime, \mu_n^h\right)$ we get that,
\begin{align*}
  \Phi_{\XX}\left(h_n(S_n), \mu_n^h \right)\leq    \Phi_{\XX_n}\left(S_n, \mu_n^h \right)+\tilde C\varepsilon_n\leq   \Phi_{\XX_n}\left(S_n^\prime, \mu_n^h\right)+ \tilde C\varepsilon_n\leq \Phi_{\XX}\left(h_n(S_n^\prime), \mu_n^h \right)+2\tilde C\varepsilon_n \;.
\end{align*}
Taking $n\to \infty$, we conclude the announced result.
\end{proof}

\subsection{Stability of Clusters}

The stability of Voronoi cells is a very natural question. However, to the best of our knowledge, it has so far been addressed only in the Euclidean setting; see Reem's paper \cite{reem2011geometric}. In \cite{garcia2020error} rates of convergence for the clusters in the $L^2$ sense are proved for the spectral clustering algorithm on a Riemannian manifold. In this section, we prove the stability of Voronoi cells with respect to the Hausdorff topology in the general setting of compact metric spaces.

Recall that when the set of $k$ centroids $S=\{b_1,\cdots,\,b_k\}$ is given, the $k$ associated clusters $V_{b_1},\cdots, V_{b_k}$ are defined as the Voronoi cells of the centroids with respect to the set $S$:
\begin{equation}\label{def:voronoicells}
    V_{b_i} := \left\{x\in\XX\,\left|\, \forall b\in S,\, d(x,b_i)\leq d(x,b)\, \right\}\right..
\end{equation}
We use $\mathcal V (S)$ for the family of Voronoi cells $\{V_{b_1},\cdots, V_{b_k}\}$.

We can now state the main result of this section, establishing the continuity of Voronoi cells with respect to their  centers. 

\begin{thm}
\label{cont_vor}
    Let $(\XX,d)$ be a compact metric space. Assume that $(S_n)_{n\geq 1}$ is a sequence of finite subsets of $\XX$ converging to some finite $S$ in the Hausdorff distance: $d_H(S_n, S)\to 0$. Then, 
$$\lim_{n\to \infty} \,\max_{V\in \mathcal V(S_n)} \,\min_{W\in \mathcal V(S)} \max_{v\in V} d(v, W) \,=0.$$
\end{thm}

This theorem is in the same spirit as Theorem \ref{thm:convergence}. It should also be interpreted as a \emph{no–false-positive guarantee}: for every $\varepsilon>0$, there exists $N\in\mathbb{N}$ such that for
all $n\geq N$, and for every cluster
$V\in \mathcal V(S_n)$, there exists a
corresponding `true' cluster $W\in \mathcal V(S)$ whose $\varepsilon-$neighborhood contains the empirical cluster, i.e. $ V\subset W^\varepsilon.$

{In what follows, we present an extension of Theorem \ref{cont_vor} to the context of clustering, contemplating Gromov-Hausdorff convergence and the existence of at most a finite numbers of  elements on $\mathcal S_{\mathcal X}(k)$. }

\begin{thm}
\label{cor:consistenceclusteringundermGH}
Let $\mathcal X_n = (\XX_n,d_n,\mu_n)$ and $\mathcal X=(\XX, d, \mu)$ be compact metric measure spaces such that $$\mathcal X_n \overset{mGH}{\to} \mathcal X.$$
Let $h_n\colon \XX_n \to \XX$ denote any sequence of $\varepsilon_n$-isometries arising in the mGH convergence of $(\mathcal X_n)_n$. Assume that $\mathcal S_\mathcal{X}(k)$ is finite. Consider the set of all Voronoi cells associated to any $h_n(S_n)$ with $S_n\in \mathcal S_{\mathcal X_n}(k)$, given by 
$$\mathcal U_n :=\bigcup_{S_n\in \mathcal S_{\mathcal X_n}(k)} \mathcal V(h_n(S_n)),$$
and the set of all Voronoi cells associated to any $S\in \mathcal S_{\mathcal X}(k)$ given by
$$\mathcal U :=\bigcup_{S\in \mathcal S_{\mathcal X}(k)} \mathcal V(S).$$
Then, 
$$\lim_{n\to \infty} \max_{V\in \mathcal U_n} \,\min_{W\in \mathcal U} \,\max_{v\in V}\, d(v, W)=0.$$
\end{thm}

\begin{rmq}
Note that distinct centroids may coincide, and consequently the number of clusters may be strictly less than $k$.
\end{rmq}

We now turn to the proof of Theorem \ref{cont_vor}. We begin by stating the following results that will be used along its proof, included at the end of this section. 
 
\begin{lemm}\label{thm:cvclusters2}
Let $(\XX,d)$ be a compact metric space, and let $B=\{b_1,\cdots,b_k\}\subset \XX$ a finite subset. We fix $b\in B$ and consider $W$, its associated Voronoi cell.
We define the enlarged Voronoi region
\[
W(\delta):=\left\{x\in \XX \,\middle|\, \forall b'\in B,\ d(x,b)\leq d(x,b')+\delta\right\}.
\]
Then it holds that 
\[
d_H(W,W(\delta)) \to 0,\quad \text{when }\,\, \delta\to 0.
\]
\end{lemm}

\begin{proof}
Assume by contradiction the existence of $\varepsilon>0$ and $\delta_n\to 0$ such that   $d_H(W,W(\delta_n))\geq \varepsilon$.  Since $W\subset W(\delta)$, 
\begin{equation}
    \label{lejos}
\varepsilon\leq d_H(W,W(\delta_n))=\max_{\tilde w\in W(\delta_n)}d(\tilde w, W)=d(\tilde w_n, W), \quad\hbox{for some $\tilde w_n \in W(\delta_n)$},
\end{equation}
meaning that 
\begin{equation}
\label{agarrados}
\forall b'\in B,\quad d(\tilde w_n ,b)\leq d(\tilde w_n ,b') + \delta_n.
\end{equation}
By compactness,  there exists $\tilde w_{n_k}\to x$, for some $x$ in $\XX$.  From \eqref{agarrados}, we conclude that 
$x\in W$, contradicting \eqref{lejos}. 
\end{proof}

\begin{lemm}\label{thm:cvclusters2b}
Let $(\XX,d)$ be a compact metric space, and let $A$ and $B$ be two finite subsets of $\XX$ with $d_H(A, B)\leq \delta/2$. Pick $a\in A$ and $b\in B$ with $d(a, b)\leq \delta/2$ and let $V$ and $W$ denote their Voronoi cells in $\mathcal V(A)$ and $\mathcal V(B)$, respectively. 
Then, 
$$V\subset W(\delta)$$
where $W(\delta)$ denotes the enlarged Voronoi region defined in Lemma \ref{thm:cvclusters2}.
\end{lemm}
\begin{proof}
  Pick $v\in V$. To show that $v\in W(\delta)$, consider  $b'\in B$ and let $a'\in A$ satisfy $d(b',a')=d(b',A)\leq \delta/2$. Then
\begin{align*}
d(v,b) &\le d(v,a)+d(a,b)\\
&\le d(v,a')+\delta/2 \\
&\le d(v,b') + d(b',a') + \delta/2\\
&\le d(v,b') +\delta,
\end{align*}
using successively: (i) the triangle inequality, (ii) $v\in V$, (iii) the triangle inequality again, and (iv) $d(a',b')\le \delta/2$.  Since $b' \in B$ was arbitrary, we conclude that  $v\in W(\delta)$ and hence $V\subset W(\delta)$.\end{proof}

\begin{proof}[Proof of Theorem \ref{cont_vor}.]
From Lemma \ref{thm:cvclusters2},  since $\mathcal V(S)$ is finite, we get that 
$$\lim_{\delta\to 0} \, \max_{W\in \mathcal V(S)} \,d_H(W, W(\delta))=0.$$
Given $\varepsilon>0$, choose $\delta$ such that 
\begin{equation}
d_H(W,W(\delta))\leq \varepsilon, \quad \hbox{for all $W\in \mathcal V(S)$ and, therefore, $W(\delta)\subset W^\varepsilon$}, 
\end{equation}
where 
\[     
W^\varepsilon=\{x\in \XX: d(x, W) \leq \varepsilon\},
\]
denotes the $\varepsilon-$neighborhood of $W$.
For such a $\delta$, consider $n_0$ with $d_H(S_n, S)\leq \delta/2$, for $n\geq n_0$. 
Now, let  $V\in \mathcal V(S_n)$ be the Voronoi cell associated to $a$ and choose $b\in S$ with $d(a, b)\leq \delta/2$. Thus, {if $W$ denotes the Voronoi cell in $\mathcal V(S)$ associated to $b$, invoking  Lemma \ref{thm:cvclusters2b}, we get that }
$$V\subset W(\delta)\subset W^\varepsilon,$$
meaning that for any $v\in V$
$$d(v, W)\leq \varepsilon,$$ which concludes the proof.
\end{proof}

We now turn to the proof of Theorem \ref{cor:consistenceclusteringundermGH}. It is based on Theorem \ref{thm:convergence} which guarantees that each subsequence of $h_n(S_n)$, with $S_n \in \mathcal S_{\mathcal X_n}(k)$ has a further convergent subsequence whose limit point belongs to $\mathcal S_{\mathcal  X}(k)$. This can then be applied with Theorem \ref{cont_vor} thanks to the assumption that $\mathcal S_\mathcal{X}(k)$ is finite.

\begin{proof}[{Proof of Theorem  \ref{cor:consistenceclusteringundermGH}. }]
Since $\mathcal S_{\mathcal X}(k)$ is finite, we get that also $\mathcal U$ is finite. Therefore, from  Lemma \ref{thm:cvclusters2}, we obtain that 
$$\lim_{\delta\to 0} \, \max_{W\in \mathcal U} \,d_H(W, W(\delta))=0.$$
Given $\varepsilon>0$, choose $\delta$ such that 
\begin{equation}
d_H(W,W(\delta))\leq \varepsilon, \quad \hbox{for all $W\in \mathcal U$ and, therefore, $W(\delta)\subset W^\varepsilon$}, 
\end{equation}
where 
\[     
W^\varepsilon=\{x\in \XX: d(x, W) \leq \varepsilon\},
\]
denotes the $\varepsilon-$neighborhood of $W$. Now, let  $V\in \mathcal V(h_n(S_n))$ be the Voronoi cell associated to some $a\in h_n(S_n)$. For $n$ large enough,  Theorem \ref{thm:convergence} guarantees the existence of a $S_\infty \in \mathcal S_{\mathcal X}(k)$ such that $d_H(h_n(S_n), S_\infty)<\delta/2 $. 
Choose $b\in S_\infty$ with $d(a, b)\leq \delta/2$. Thus, if we denote by $W\in \mathcal V(S_\infty)$ the Voronoi cell associated to $b$, invoking  Lemma \ref{thm:cvclusters2b}, can procede following the final steps of the proof of  Theorem \ref{cont_vor} to conclude that 
$V\subset W(\delta)\subset W^\varepsilon $, 
meaning that for any $v\in V$
$d(v, W)\leq \varepsilon,$ which concludes the proof.
\end{proof}

\section{Applications}
\label{sec:applications}

In this section, we apply our consistency results, Theorem \ref{thm:convergence} and Corollary \ref{cor:consistenceclusteringundermGH}, to prove that empirical {\it k-}means and their clusters do converge for some well-known metric learning procedures. We also prove convergence of $k$-means in other contexts with no empirical data involved like first passage percolation and discretizations of continuous length-spaces.

We start by considering the general case in which the space $\XX_n$ is given by an i.i.d. sample.

\subsection{Uniformly consistent distances from i.i.d. samples}
\label{ssec:iid}

The first important application of Theorem~\ref{thm:convergence} concerns the general case of an i.i.d.\ sample $\mathbb{X}_n = \{X_1, \dots, X_n\}$ drawn from a measure $\mu$ on a metric space $(\mathbb{X}, d)$.  
More precisely, assume that $X_i : \Omega \to \mathbb{X}$ is a sequence of i.i.d.\ random variables defined on a probability space $(\Omega, \mathcal{A}, P)$. By considering the empirical measure $\mu_n$, i.e., the uniform distribution on $\mathbb{X}_n$, we obtain a random metric measure space $(\mathbb{X}_n, d_n, \mu_n)$, where the distances $d_n$ can be chosen in various ways. 
In this context, it is natural to consider, for all $n$, the inclusion function $h_n(x) = x$ as a family of candidates for the $\varepsilon_n$-isometries required for mGH convergence. 

In this setting, since the empirical measures converges weakly to the population one, the mGH convergence of $(\mathbb{X}_n, d_n, \mu_n)$ towards $(\mathbb{X}, d, \mu)$ is reduced to verifying that the immersions $h_n\colon \mathbb X_n \to \mathbb X$ give in fact a family of $\varepsilon_n$-isometries with $\varepsilon_n \to 0$. That is the content of the following proposition.

\begin{prop}\label{prop:mGHsample}
Consider $(X_i)_{i\geq 1}$ i.i.d. random elements taking values in the compact space $(\XX,d)$, all distributed according to $\mu$. Then, the  sample metric measure spaces $(\mathbb{X}_n, d_n, \mu_n)$
converge almost surely to $(\mathbb{X}, d, \mu)$ in the measured Gromov-Hausdorff
topology, provided that the following two conditions hold almost surely as $n\to\infty$:
\begin{enumerate}[i)]
 \item $$\underset{x,x' \in \XX_n}{\sup}\, |d_n(x,x') - d(x,x')| \to  0,$$
 \item $$\underset{x \in \XX}{\sup}\,\,\,d(x,\XX_n) \to 0.$$
\end{enumerate}
\end{prop}

\begin{proof}
These two conditions imply immediately that $h_n\colon \mathbb X_n \to \mathbb X$, $h_n(x)=x$, are $\varepsilon_n$-isometries with $\varepsilon_n \to 0$. The weak convergence of $(h_n)_\#(\mu_n) \to \mu$ is obtained from the a.e. weak convergence of the empirical measures $\mu_n=n^{-1}\sum_{i=1}^n \delta_{X_i}$ to $\mu$. 
\end{proof}

\begin{rmq}
Note that in the case of a compact Riemannan manifold the second condition is satisfied as long as $\mu$ is absolutely continuous with respect to the volume measure and has full support. As a result, the mGH convergence for the random metric measure spaces induced by an i.i.d.\ sample from a fully supported measure on a compact Riemannian manifold reduces in practice to the uniform convergence of the distances.
\end{rmq}

\subsection{Fermat and Isomap distances}
\label{ssec:fermat}

Isomap \cite{Isomap} and Fermat distances \cite{groisman2022nonhomogeneous, hwang2016, MCD, SGJ} have been introduced as powerful tools to estimate intrinsic distances on nonlinear data structures embedded in high-dimensional spaces.

Isomap (Isometric Mapping) was originally introduced as a nonlinear dimensionality reduction technique that aims to preserve the intrinsic geometric structure by approximating geodesic distances through neighborhood graphs. Unlike traditional linear methods such as PCA, Isomap captures the underlying manifold structure by constructing a graph where each point is connected to its nearest neighbors, with edges weighted by Euclidean distances (or by constructing any other graphs in which each node is connected to points that are close to it). The geodesic distance is then estimated as the shortest length path in this weighted graph. Typically, these distances are then used as input to multidimensional scaling (MDS) for embedding data into lower-dimensional spaces while preserving global geometric relationships.

Building upon the geodesic distance concept, Fermat distances were introduced as a generalization that considers paths minimizing a power-weighted length rather than the traditional shortest path used in Isomap and thus taking into account the intrinsic density of data points. While geodesic distances focus solely on the shortest path along the manifold, Fermat distances introduce a parameter $\alpha \geq 1$ that controls the weighting of path lengths. For $\alpha = 1$, the Fermat distance constructed on a graph of nearest-neighbor points coincides with the Isomap geodesic distance estimator, while for $\alpha > 1$, it favors paths that are not necessary short in length but avoid large jumps between consecutive points of the path, thus favoring high-density regions. This makes Fermat distances particularly suitable for irregular data distributions.

It has been shown empirically that performing $k$-means and other clustering procedures using these data-driven distances significantly outperforms the standard $k$-means algorithm based on Euclidean distances \cite{MCD, SGJ, Chazal, Little, Trillos} and gives results at the state-of-the-art of most successful methods. In particular, \cite{Chazal} demonstrates that, at least at the population level, this approach allows the detection of clusters with arbitrary geometries. In light of this, it is important to understand whether the centers and clusters computed at the empirical level are consistent estimators of their population counterparts.

In this subsection, we study the convergence of empirical barycenters (respectively, {\em k}-barycenters) computed with respect to the empirical Fermat distance, towards the population (respectively, $k$-) barycenters defined with respect to the population Fermat distance. In the same spirit, we study the barycenters computed with Isomap distances.

Let $X_1, \dots, X_n$ be a sample of $n$ i.i.d.\ points with density $f$ with respect to the volume measure on the $\ell$-dimensional manifold $(\M,g)$, i.e. $\int_\M f(x)\,d\mathrm{vol}_g(x) =1, $ where in local coordinate $d\mathrm{vol}_g(x) = \sqrt{\det g}\,dx$ stands for the volume measure on $(\M,g)$.

We may assume that $(\M,g)$ is embedded in an Euclidean space $\R^D$ of large enough dimension.

\subsubsection{Fermat distances}

For $\alpha > 1$, the \emph{empirical Fermat distance} is defined as:
\[
d_{n,\alpha}(x,y) = \inf_\gamma \sum_{i=0}^r |X_{i+1} - X_i|^\alpha,
\]
where $|\cdot |$ denotes Euclidean norm in $\R^D$ and the infimum is taken over all paths $\gamma = (X_0, \dots, X_{r+1})$ with $X_0 = x$, $X_{r+1} = y$.

Note that we are considering any possible path in the complete graph. Isomap distances are constructed similarly but with $\alpha=1$. In that case it is necessary to impose to the paths that two consecutive points on them have to be nearest neighbors (or any other condition that guarantees that they are close to each other in Euclidean distance), see \cite{Isomap}. When $\alpha >1$ (Fermat distances) it is not necessary to impose this restriction since optimal paths prefer to do that automatically \cite{groisman2022nonhomogeneous}. 

For the population version, we consider a probability measure $\mu$ in $(\M,g)$ with density $f$. The {\em population }Fermat distance $d_{f,\alpha}$ associated to $(\M,g, \mu)$ is the Riemannian distance induced by the conformal change of metric $\tilde{g} = f^\kappa\, g$, where $\kappa = (1-\alpha)/\ell.$ In other words, for $x,y\in\M$
\[
d_{f,\alpha}(x,y) = \inf_\gamma \int_I f^{(1-\alpha)/d}(\gamma_t) |\dot \gamma_t| dt.
\]
The infimum is taken over all piecewise smooth curves $\gamma\colon {I=}[0, 1]\to\M$ with $\gamma(0)=x$ and $\gamma(1)=y$.

\noindent When the points $X_i$ are i.i.d. drawn from $f$, the empirical distances converge to the population one in the sense that there exists a constant $\mathsf  c > 0$, depending only on the parameter $\alpha$ and the dimension $\ell$ of $\M$, such that, almost surely,
\begin{equation}\label{eq:cvFermat}
    n^{{(\alpha - 1)}/{\ell}}\, d_{n,\alpha} \underset{n \to \infty}{\longrightarrow} \mathsf  c\, d_{f,\alpha}.
\end{equation}

This convergence has been established pointwise when $\M$ is isometric to the closure of an open subset of some $\mathbb{R}^\ell$ in \cite{groisman2022nonhomogeneous} and strengthened to uniform convergence for closed (compact and boundaryless) smooth manifolds \cite{FBMG-JMLR23}.

We consider $\mathcal S_{\mathcal X_n}(1)$, the set of Fréchet means of $\mathcal X_n=(\mathbb{X}_n, d_{n,\alpha}, \mu_n)$ with respect to the empirical Fermat distance:
\[
    \mathcal S_{\mathcal X_n}(1) = \underset{x \in \mathbb{X}_n}{\argmin} \sum_{i=1}^n d_{n,\alpha}(x, X_i)^p.
\]
Here  $\XX_n= \{X_1, \dots, X_n\}$ and we implicitly take $\mu_n$ to be the uniform measure on $\mathbb{X}_n$.
Similarly, $\mathcal S_{\mathcal X}(1)$ denotes the population Fréchet mean of $\mathcal X$ with respect to the population Fermat distance (i.e., $\mathcal X=(\XX, d_{f,\alpha},\mu)$ with $\mu$ being the measure with density $f$ on $\M$), 
\[
\mathcal S_{\mathcal X}(1) = \underset{x \in \M}{\argmin}\, \int_{\M} d_{f,\alpha}(x, y)^p \mu(dy).
\]
More generally, for any $k \in \mathbb{N}$, we consider the empirical $k$-means $\mathcal S_{\mathcal X_n}(k)$,
\[
\mathcal S_{\mathcal X_n}(k) = \underset{S \in C_k(\mathbb{X}_n)}{\argmin} \sum_{i=1}^n d_{n,\alpha}(S, X_i)^p,
\]
and the population $k$-means,
\[
\mathcal S_{\mathcal X}(k) = \underset{S \in C_k(\M)}{\argmin}\, \int_{\M} d_{f,\alpha}(S, y)^p \mu(dy).
\]

\noindent In view of the above discussion, Theorem~\ref{thm:convergence} applies, yielding the following result.

\begin{prop}
\label{Convergence.Fermat.Barycenters}
Let $\alpha > 1$. Assume that $\M$ is a smooth, closed (i.e., compact and without boundary) $\ell$-dimensional manifold embedded in $\mathbb{R}^D$, and let $f\colon \M \to \mathbb{R}_{>0}$ be a smooth density function. Then, almost surely,
\begin{equation*}
\lim_{n \to \infty}\, \max_{\hat b \in \mathcal S_{\mathcal X_n}(k)}\,\, \min_{b \in \mathcal S_{\mathcal X}(k)}\, d_H(\hat b, b) = 0,
\end{equation*}
where $d_H$ denotes the Hausdorff distance between compact subsets of $\M$ with respect to $d_{f,\alpha}$.
If $\mathcal S_{\mathcal X}(k)$ is a singleton, this convergence can be written as
\begin{equation*}
\lim_{n \to \infty}\, \max_{\hat b \in \mathcal S_{\mathcal X_n}({k})}\, d_H(\hat b, \mathcal S_{\mathcal X}(k)) = 0 \quad \text{a.s.}
\end{equation*}
If, in addition, $\mathcal S_{\mathcal X_n}(k)$ is also a singleton, $\mathcal S_{\mathcal X_n}(k) = \{b_n^k\}$, then
\begin{equation*}
\lim_{n \to \infty} d_H(b_n^k, \mathcal S_{\mathcal X}(k)) = 0 \quad \text{a.s.}
\end{equation*}
{For $k = 1$, we have
\begin{equation*}
\lim_{n \to \infty} d_{f,\alpha}(b_n^1, \mathcal S_{\mathcal X}(1)) = 0 \quad \text{a.s.}
\end{equation*}}
\end{prop}

\begin{rmq}
Proposition \ref{Convergence.Fermat.Barycenters} states that the barycenter computed with the empirical measure converges almost surely to the population barycenter in $d_{f,\alpha}$ distance. Since $f$ is bounded,  
$\lim_{n \to \infty} d_{f,\alpha}(b_n, \beta) = 0$ implies $\lim_{n \to \infty} d_{\M}(b_n, \beta) = 0$. Here $d_\M$ is the geodesic distance on the manifold. Hence we also have convergence in this sense.   
\end{rmq} 

\begin{rmq}
Proposition \ref{Convergence.Fermat.Barycenters} states the consistency of the centroid calculated using the {\it k-}mean procedure under a metric learned from the Fermat distance, by applying Theorem \ref{thm:convergence}. 
Note that, similarly, Corollary \ref{cor:consistenceclusteringundermGH} also applies and gives the consistency of clusters  calculated using the $k$-mean procedure under a metric learned from the Fermat distance.
\end{rmq}

\subsubsection{Isomap distances} A similar result holds for the geodesic distance estimator involved in the Isomap algorithm \cite{isomapproofs} since, as with the Fermat distances, these estimators converge uniformly. 

We will focus on the $\varepsilon$-graph in which two points $X_i$ and $X_j$ are adjacent if and only if $|X_i-X_j| \le \varepsilon$. In this graph we assign the weight $w_{ij}=|X_i-X_j|$ to the edge $\{X_i, X_j\}$ when $X_i \sim X_j$. 

Given $\varepsilon>0$, the \textit{Isomap distance} is defined as the $\varepsilon$-graph distance between $X_i$ and $X_j$, i.e.
\[
d_{\XX_n,\varepsilon}(x,y) = \inf_\gamma \sum_{i=0}^r |X_{i+1}- X_i|,
\]
where the infimum is taken over all paths $\gamma=(X_0, \dots, X_{r+1})$ with $X_0=x$, $X_{r+1}=y$, $X_i\in \mathbb X_n$ for every $1\le i \le r$ and $|X_{i+1}-X_i|\le \varepsilon$ for every $0\le i \le r+1$.

There is an alternative version of the Isomap distance in which instead of considering the $\varepsilon$-graph, the $k$-nearest neighbors graph ($k$-NN) is considered. In this graph there is an edge between $X_i$ and $X_j$ if one of them is among the $k$ nearest points of the other with respect to the Euclidean distance. The following theorem can be deduced from \cite{isomapproofs}. See also \cite[Theorem 1, Theorem 9]{Alamgir-vonLuxburg}  

\begin{thm}
\label{thm:isomap}
Let $(X_i)_{i=1,\ldots,n}$ be i.i.d. with (smooth) density $f\colon \M \to \R_{>0}$ with respect to volume measure on $\mathcal M$. Assume $\varepsilon_n \to 0$ and 
$n \varepsilon_n^d/\log n \to \infty$. Then,
\[
\lim_{n\to \infty} \sup_{x,y}\left| d_{\XX_n, \varepsilon_n}(x,y) - d_{\mathcal M}(x,y )\right|=0,\quad\hbox{almost surely}. 
\]
Here, $d_{\M}$ denotes the inherited geodesic distance.
\end{thm}

Note that in \cite{isomapproofs}, some additional assumptions on $\mathcal M$ are made, but they are not necessary for the convergence of the estimator $d_{\XX_n, \varepsilon_n}$ to $d_{\M}$. 

In view of this, Proposition \ref{Convergence.Fermat.Barycenters} holds also when Fermat distances are replaced with Isomap distances in the regime given in Theorem \ref{thm:isomap}.

\subsubsection{A uniqueness criterion for Fréchet means under the Fermat distance}
\label{NPC}

As shown in \cite{Cuevas}, the uniqueness of the set of $k$-means is a strong indicator of an appropriate choice of $k$. Indeed, if several distinct sets of $k$-means exist, this strongly suggests that $k$ does not correspond to the true number of clusters in the data. Furthermore, uniqueness is closely related to the algorithm’s stability (see \cite{von2010clustering}). In light of this, it is important to understand the conditions under which uniqueness holds.
In this section, we take a first step toward a theoretical understanding of the uniqueness issue within our metric learning framework. We introduce a criterion, based on the non-positive curvature of the Fermat metric, that guarantees the uniqueness of the Fréchet mean (\textit{i.e.}, the case $k=1$ and $p=2$ in the notation of Section \ref{sec:continuitykmean}) computed with respect to this metric. In other words, when this criterion is satisfied, it indicates that the data are not partitioned into several groups, but rather form a single cluster.

A metric space $(X,d)$ with non-positive curvature is one in which triangles are thinner than their Euclidean counterparts. This simple comparison property implies that, in such spaces, every probability distribution with a finite first moment admits a unique Fréchet mean, just as in the Euclidean setting, which serves as the comparison model. One of the first works linking the uniqueness of the Fréchet mean and non-positive curvature is \cite{Kendall1990}. This was later refined in \cite{sturm2003probability} and \cite{Bacak2014}. Note also that such a theory has been developed for spaces with curvature bounded above, even by some positive constant, in which case an additional assumption of small support is needed to guarantee uniqueness; see \cite{Bhattacharya}. More recent works have further shown that in these spaces, standard algorithms for computing Fréchet means also perform well. For instance, the empirical barycenter of a sample of a bounded random variable is sub-Gaussian, thus extending Hoeffding’s inequality to this nonlinear context (see \cite{brunel2024concentration} for non-positively curved spaces, and \cite{brunel2025finite} for the general case of spaces with curvature bounded above).
In this section, we provide a condition ensuring that a Fermat metric space is non-positively curved, and therefore enjoys barycentric properties as favorable as those of the Euclidean space.

Recall from the previous section that $\alpha>1$, and the Fermat metric is equal to the conformal change of metric $\tilde{g}=f^\kappa\,g$, with $\kappa=(1-\alpha)/l<0$, $f:\M\to\R_+$ the probability density of the model, and $l$ the intrinsic dimension of $\M$.

\begin{thm}\label{thm:uniqueness}
Assume that $\M$ is simply connected, and that $X$ admits a finite first moment with respect to the Fermat distance, i.e. $\E\,d_{f,\alpha}(X,x_0) <\infty$ for some $x_0\in \M$.
Assume also that for all vectors $u,v$ orthonormal with respect to the metric $g$, the following inequality
\[
\frac{\kappa}{2f}\left( \nabla^2f(u,u) + \nabla^2 f(v,v) \right) + \frac{\kappa^{2}}{4f^{2}}|\nabla f|^2 - \frac{\kappa}{2f^{2}}\left( \nabla f \cdot u \right)^2 \left( 1+\frac{\kappa}{2}\right) - \frac{\kappa}{2f^{2}}\left( \nabla f \cdot v \right)^2 \left( 1+\frac{\kappa}{2}\right) \geq \mathsf K(u,v)
\]

is satisfied by the density $f$ and the sectional curvatures $\mathsf K$ of the manifold $(\M,g)$. Then the random variable $X$, seen as taking values in $(\M,\tilde{g})$ equipped with the Fermat metric $\tilde{g}$, admits one and only one Fréchet mean.
\end{thm}

Note that a priori the variable $X$ can have more than one Fréchet mean with respect to the original metric $g$, but uniqueness appears in the Fermat metric.

\begin{proof}
Let us first show that under the assumption of the theorem, the Fermat manifold $(\M,\tilde{g})$ is Hadamard, i.e. it has non-positive sectional curvatures. To this aim, we will compute the sectional curvature $\tilde{\mathsf K}$ of the Fermat manifold by using the usual formula for a conformal change of metric. Recall that if $(\M,g)$ is a Riemannian manifold, $\psi:\M\to\R$ a $C^2$ function, $\tilde{g}=e^{2\psi}g$ a conformal change of metric, $\mathsf K$ the sectional curvature of $g$, and $\tilde{\mathsf K}$ the sectional curvature of $\tilde{g}$, then for all vectors $u,v$ orthonormal with respect to $g$, it holds (see e.g. \cite{doCarmo})
\begin{equation}\label{eq:conformalsect}
e^{2\psi}\tilde{\mathsf K}(u,v) = \mathsf K(u,v) -\nabla^2\psi(u,u) -\nabla^2\psi(v,v) - g(\nabla\psi,\nabla\psi) + (u(\psi))^2 + (v(\psi))^2
\end{equation}
Note that the gradient and the hessian are computed in the metric $g$, and $u(\psi) = g(\nabla \psi, u)$.
Now, we apply Formula \eqref{eq:conformalsect} with $\psi = \frac{\kappa}{2}\log(f)$. So $\nabla \psi = \frac{\kappa}{2f}\nabla f$ and $\nabla^2 \psi = \frac{\kappa}{2f}\nabla^2 f -\frac{\kappa}{2f^2}\nabla f\otimes \nabla f$. Hence, denoting $\tilde{\mathsf K}$ for the sectional curvature of $(\M,\tilde{g})$, we get
\begin{align*}
\tilde{\mathsf K}(u,v) &= e^{-2\psi} \left\{\mathsf K(u,v) -\nabla^2\psi(u,u) -\nabla^2\psi(v,v) - g(\nabla\psi,\nabla\psi) + (u(\psi))^2 + (v(\psi))^2 \right\}\\
&= f^{-\kappa} \left\{ \mathsf K(u,v) - \frac{\kappa}{2f}\nabla^2 f(u,u) +\frac{\kappa}{2f^2}(\nabla f\otimes \nabla f)(u,u)  - \frac{\kappa}{2f}\nabla^2 f(v,v) +\frac{\kappa}{2f^2}(\nabla f\otimes \nabla f)(v,v)\right.\\
&- \left. \frac{\kappa^2}{4f^2}|\nabla f|^2 + \frac{\kappa^2}{4f^2} g(\nabla f,u)^2 + \frac{\kappa^2}{4f^2} g(\nabla f,v)^2 \right\} \\
& = f^{-\kappa} \mathsf K(u,v) - \frac{\kappa}{2f^{\kappa+1}}\nabla^2 f(u,u) - \frac{\kappa}{2f^{\kappa+1}}\nabla^2 f(v,v) - \frac{\kappa^2}{4f^{\kappa+2}}|\nabla f|^2\\
&+\frac{\kappa}{2 f^{2+\kappa}} g(\nabla f, u)^2 \left(1+\frac{\kappa}{2} \right) + \frac{\kappa}{2 f^{2+\kappa}} g(\nabla f, v)^2 \left(1+\frac{\kappa}{2} \right)
\end{align*}
where we used that $(\nabla f\otimes \nabla f)(v,v) = g(\nabla f, v)^2$. Then we easily deduce that under the assumption of the theorem, the Fermat manifold is Hadamard, that is $\tilde{\mathsf K}\leq 0$. Then it is a famous result that every probability measures admitting a finite first moment on a non-positively curved, simply connected manifold admits a unique Fréchet mean, see \cite[Theorem 2.1]{Bhattacharya}. The proof is complete.
\end{proof}

As a corollary, one can deduce the uniqueness of the Fréchet mean in the Fermat distance associated to any log-concave distribution in $\R^\ell$. In particular, uniqueness holds for the Gaussian distribution.

\begin{corol}\label{cor:uniqlogconcave}
Let $(\M,g)=(\R^\ell,|\cdot|)$ be the Euclidean space, $X$ be a random variable with a log-concave density $f$ with respect to the Lebesgue measure, and let $d_{f,\alpha}$ be the associated Fermat distance for some $\alpha>1$. Then $X$ admits a unique Fréchet mean for the Fermat distance, as soon as it admits a finite first moment $\E \left[d_{f,\alpha}(X,x_0) \right] < \infty$, for some $x_0\in\R^\ell$.
\end{corol}

\begin{proof}
The result follows from a direct application of Theorem \ref{thm:uniqueness}, by noting that for orthonormal vectors $u,v$, the Cauchy-Schwarz inequality gives $\langle \nabla \log f, u\rangle^2 + \langle \nabla \log f, v \rangle^2 \leq |\nabla \log f|^2$. 
\end{proof}
Let us conclude this section by an illustration of Corollary \ref{cor:uniqlogconcave} on the Gaussian distribution. To simplify, we do the computation in dimension one. Let $f(x)=(2\pi)^{-1/2} e^{-x^2/2}$ be the Gaussian density. The conformal Fermat metric is then given by $\tilde{g} = f^\kappa = (2\pi)^{-\kappa/2} e^{-\kappa x^2/2} $, $\kappa = (1-\alpha)/\ell = 1-\alpha$, since we are in dimension $l=1$, and the Fermat distance between $x,y$ is given by
\[
d_{f,\kappa}(x,y) = \int_0^1 \sqrt{\tilde{g}\left(\dot{\gamma}(t),\dot{\gamma}(t)\right)} dt = (2\pi)^{-\kappa/4}|y-x| \F(x,y)
\]
where $\gamma(t) = (1-t)x + ty$ and $$ \F(x,y) = \int_0^1 e^{-\frac{\kappa}{4}((1-t)x+ty)^2} dt.$$

So up to constant, we have that if $Z$ is a standard normal distribution,
\[
\E\, d_{f,\kappa}(0,Z) = \int_{\R} (2\pi)^{-\kappa/4}|y|\F(0,y) (2\pi)^{-1/2}e^{-y^2/2} dy \sim \int_\R |y|\int_0^1 e^{-\frac{\kappa}{4}y^2t^2} dt\,e^{-\frac{y^2}{2}}dy
\]
And an elementary study gives that this integral is finite if, and only if, $\alpha\in (-1,3)$.
Therefore we conclude that the normal variable $Z$ admits a finite first moment in its own Fermat distance if, and only if the parameter $\alpha\in[1,3)$, whence from Corollary \ref{cor:uniqlogconcave} it admits a unique Fréchet mean on this range. Note that on the range $\alpha\geq 3$, it is not integrable, and hence it does not admit a Fréchet mean.

\subsection{Diffusion distances}\label{sec:spectraldistances}

Among nonlinear spectral dimensionality reduction techniques, \emph{Laplacian eigenmaps} and \emph{diffusion maps}, introduced respectively in \cite{belkin2003laplacian, coifman2006diffusion}, are perhaps the most well-known methods. Both rely on the idea of embedding the data into a lower-dimensional Euclidean space  using spectral coordinates derived from a discrete Laplace operator that encodes pairwise similarities between data points.
Given a dataset $\XX_n = \{X_1, \ldots, X_n\} \subset \XX \subset \R^{\ell}$, one first defines a similarity matrix $(\eta_{ij})_{ij}$, where each entry $\eta_{ij} = \eta(X_i, X_j)$ is computed from a positive, continuous, and symmetric kernel \mbox{$\eta : \XX \times \XX \to \R_+$}. In practice, this kernel is often taken to be Gaussian,
\begin{equation}\label{def:Gaussiankernel}
\eta(x,y) = \exp \left(-\frac{|x - y|^2}{2\sigma^2}\right).
\end{equation}
The \emph{degree matrix} $D$ is then defined as the diagonal matrix with entries $d_i = \sum_{j=1}^n \eta_{ij}$. From these definitions, one constructs the \emph{normalized graph Laplacian}
\begin{equation}\label{def:graphLaplace}
L = I - D^{-1/2} \eta D^{-1/2}
\end{equation}
where $I$ denotes the $n \times n$ identity matrix and by abuse of notation we used $\eta$ also for the matrix with entries $\eta_{ij}$. This defines a quadratic form on $\R^{n}$, given for any $Y=(y_1, \ldots, y_n) \in \R^n$ by

\[
Y^T L Y = \frac{1}{2} \sum_{i,j=1}^n \eta_{ij} \left( \frac{y_i}{\sqrt{d_i}} - \frac{y_j}{\sqrt{d_j}} \right)^2.
\]
Next, it is straightforward to verify, using the normalization by $\sqrt{d_i}$, that the spectrum of $L$ is contained in $[0,1]$. One can then compute the first $k$ eigenvectors $u_1, \ldots, u_k \in \R^n$ corresponding to the $k$ smallest eigenvalues
\[
0 = \lambda_1 \leq \cdots \leq \lambda_k \leq 1 
\]
of the eigenproblem
\[
L u = \lambda u.
\]
The idea is then to use these eigenvectors
\begin{align*}
u_1 &= (u_1^1, \ldots, u_1^n)\in\R^n\\
&\vdots\\
u_k &= (u_k^1, \ldots, u_k^n)\in\R^n
\end{align*}
to embed the dataset $\XX_n$ into $\R^k$. The Laplacian eigenmaps and diffusion maps methods differ in how they define the embedding. The Laplacian eigenmaps approach uses the eigenvectors directly, via the map
\begin{align*}
\varphi \colon \XX_n &\to \R^k\\
X_i &\mapsto (u_1^i, \ldots, u_k^i).
\end{align*}
In contrast, diffusion maps introduce a \emph{diffusion parameter} $t>0$, defining the embedding
\begin{align*}
\varphi_t : \XX_n &\to \R^k\\
X_i &\mapsto ((1-\lambda_1)^t\, u_1^i, \ldots, (1-\lambda_k)^t\, u_k^i).
\end{align*}
Note that $\varphi_t$ converges to $\varphi$ as $t \to 0$. 
The standard spectral clustering procedure consists of applying the $k$-means algorithm to the embedded points $ \varphi(X_i) \in \R^k$, which yields $k$ clusters $A_1, \ldots, A_k \subset \R^k$.  
One may also work with the diffusion map $\varphi_t$ for any $t>0$, in which case the embedded points are $ \varphi_t(X_i)$ and the $k$-means algorithm produces clusters
\[
A_{1,t}, \ldots, A_{k,t} \subset \R^k.
\]
The corresponding clusters in the original dataset are then given by
\[
\varphi_t^{-1}(A_{1,t}), \ \ldots,\ \varphi_t^{-1}(A_{k,t}) \subset \XX_n.
\]
Since $k$-means is performed with respect to the Euclidean distance in $\R^k$, it is natural, if one wishes to perform clustering directly on the dataset $\XX_n$, to consider the pull-back of the Euclidean distance via $\varphi_t$, defined by
\begin{align*}
d_t : \XX_n \times \XX_n &\to \R_+\\
(X_i, X_j) &\mapsto \left\| \varphi_t(X_i) - \varphi_t(X_j) \right\|_2.
\end{align*}
Note that this only defines a pseudo-distance, since $\varphi_t$ is usually non-injective, implying that it is possible to have $d_t(x,y)=0$ but $x\neq y$. This pseudo-distance $d_t$ is however referred to as the \emph{truncated diffusion distance}, and depends on $n$. We will omit this dependence in the notation for simplicity. It satisfies the explicit formula
\begin{equation}\label{def:diffusiondistance}
d_t(X_i, X_j)^2 = \sum_{l=1}^k (1-\lambda_l)^{2 t}\, (u_l^i - u_l^j)^2.
\end{equation}
The true (non truncated) diffusion distance corresponds to the infinite sum over all eigenvalues, and it is really a distance function. 

{The reader should keep in mind that in order to be fully rigorous, we have to take the quotient $\XX_n/\sim$ where data points having the same image by the diffusion map $\varphi_t$ are identified, in order for the truncated distance to be a distance function. Since this would only complicate notations without make things clearer, in the following we do not mention it anymore.}

Unlike Isomap or Fermat distances, which compute a (weighted) $\varepsilon$-graph geodesic distance by \emph{minimizing} the total weight among all paths connecting two points, the diffusion distance $d_t$ measures the discrepancy between the distribution at time $t$ of two {random walks} running on sample points $\XX_n$ but started at $X_i$ and $X_j$ respectively. The transitions are proportional to $\eta_{ij}$.

By construction, the diffusion map $\varphi_t$ provides (after taking quotient) an isometry between $(\XX_n,d_t)$ and $(\varphi_t(\XX_n),|\cdot |) \subset (\R^k,|\cdot | )$ that preserves the uniform measure and one can expect performing $k$-means in both spaces to be equivalent. Since standard spectral clustering applies the $k$-means algorithm to the embedded points $\varphi_t(X_i)$ using the Euclidean distance in the ambient space $\R^k$, this approach differs conceptually slightly from performing $k$-means in $\mathcal{X}_n = \left( \XX_n,\, d_t,\, \mu_n \right)$. 

Consistency of spectral clustering has been widely studied. We return to this point by the end of the subsection.

With this perspective, we can recover the consistency of the $k$-means clustering procedure on $\mathcal{X}_n$ using Theorem~\ref{thm:convergence} once we prove that the sequence of random metric measure spaces $\mathcal{X}_n$ converges almost surely in the measured Gromov–Hausdorff topology to some limiting metric measure space $\mathcal{X}$.

For a fixed $t \geq 0$, suppose that $\XX_n$ is an i.i.d.\ sample drawn from a probability measure $\mu$ supported on a closed manifold $\mathcal{M} \subset \R^{\mathsf D}$.

According to Section~\ref{ssec:iid}, almost sure convergence in the measured Gromov–Hausdorff topology will hold provided that the diffusion distance $d_t$ converges almost surely and uniformly to some distance $\delta_t$ on $\mathcal{M}$ as $n \to \infty$.
In other words, we seek to show the existence of a distance $\delta_t$ on $\mathcal{M}$ such that
\[
\text{almost surely,} \quad 
\sup_{i,j} \left| d_t(X_i, X_j) - \delta_t(X_i, X_j) \right|
\underset{n \to \infty}{\longrightarrow} 0.
\]
From \eqref{def:diffusiondistance}, this convergence will occur if, simultaneously, the first $k$ eigenvalues $\lambda_1 \leq \cdots \leq \lambda_k$ of $L$ converge, that is,
\[
\forall i = 1, \ldots, k, \quad 
\lambda_i \underset{n \to \infty}{\longrightarrow} \sigma_i,
\]
for some limiting sequence $\sigma_1 \leq \cdots \leq \sigma_k$, and if the associated eigenvectors $u_1, \ldots, u_k$ also converge. More precisely, for each $i = 1, \ldots, k$, there should exist a function $f_i : \mathcal{M} \to \R$ such that
\[
\text{almost surely,} \quad 
\sup_{j = 1, \ldots, n} \big| u_i^j - f_i(X_j) \big| 
\underset{n \to \infty}{\longrightarrow} 0.
\]
It turns out that both types of convergence have been established in the literature under mild assumptions, and we recall below a result from \cite{von2008consistency} that is sufficient for our purposes. 
A broader discussion of the literature on the convergence of graph Laplacians will be provided afterwards.

\begin{thm}[{\cite[Theorem 15]{von2008consistency}}]\label{thm:spectralconvergenceliterature}
Let $\M$ be a compact manifold,
Consider $(X_i)_{i\geq 1}$ i.i.d. random elements taking values in $\M$, all distributed according to $\mu$ with full support on $\M$.  
Define  $\XX_n = \{X_1, \ldots, X_n\}$.
Assume that the similarity kernel $\eta : \XX \times \XX \to \R_+$ is symmetric, continuous, and strictly positive, i.e.,
\[
\forall x,y \in \M, \quad \eta(x,y)  > 0.
\]
Then the graph Laplacian operator $L$ defined by \eqref{def:graphLaplace} converges compactly to a linear operator $U$ on $\M$ {almost surely}. Moreover, if the first $k$ eigenvalues $0 \leq \sigma_1 \leq \cdots \leq \sigma_k \leq 1$ of $U$ are simple and satisfy $\sigma_k \neq 1$, then, for sufficiently large $n$, the first $k$ eigenvalues of $L$ also satisfy
\[
0 = \lambda_1 \leq \cdots \leq \lambda_k < 1,
\]
with 
\[
\lambda_i \underset{n \to \infty}{\longrightarrow} \sigma_i, \quad i = 1, \ldots, k.
\]
Furthermore, each eigenvector $u_i$ associated to $\lambda_i$ converges uniformly to an eigenfunction $f_i : \M \to \R$ associated to $\sigma_i$, in the sense that
\[
\sup_{j=1,\ldots,n} \big| u_i^j - f_i(X_j) \big| \underset{n \to \infty}{\longrightarrow} 0, \quad i = 1, \ldots, k.
\]
\end{thm}

For the definition of compact convergence of operators, we refer the reader to \cite[Section 4]{von2008consistency} and the references therein.
Note that the limiting operator can be expressed as
\begin{equation}\label{def:limitLaplacian}
U(f)(x) = f(x) - \mathbb{E} \left[ \frac{\eta(x,X_1)}{\sqrt{\mathsf d(x)\, \mathsf d(X_1)}}\, f(X_1) \right]
\end{equation}
where
\[
\mathsf d(x) := \mathbb{E}[\eta(x,X_1)].
\]

We emphasize that the simplicity assumption on the eigenvalues of $U$ is not overly restrictive, since simplicity is a \emph{generic} property: in other words, almost all operators have simple eigenvalues (consider, for instance, the case of a random matrix). For more details, we refer the reader to \cite{Albert75, CARLSON1979472}.

From Theorem~\ref{thm:spectralconvergenceliterature}, we obtain the uniform convergence of the diffusion distances.

\begin{corol}\label{cor:convergencediffusiondistance}
Under the same assumptions as in Theorem~\ref{thm:spectralconvergenceliterature}, for every $t \ge 0$, the diffusion distance $d_t$ defined in~\eqref{def:diffusiondistance} converges almost surely and uniformly toward a distance $\delta_t$ on $\M$:
\[
\text{almost surely,} \qquad
\sup_{i,j} \left| d_t(X_i, X_j) - \delta_t(X_i, X_j) \right|
\underset{n \to \infty}{\longrightarrow} 0.
\]
Moreover, the limit distance $\delta_t$ is given by
\[
\delta_t(x,y)^2
   = \sum_{l=1}^k (1-\sigma_l)^{2t}\, \bigl(f_l(x) - f_l(y)\bigr)^2,
\]
where $\sigma_l$ and $f_l$ denote the first $k$ eigenvalues and the corresponding eigenfunctions of the limiting operator~\eqref{def:limitLaplacian} provided by Theorem~\ref{thm:spectralconvergenceliterature}.
\end{corol}

As stated in Section~\ref{ssec:iid}, we are therefore in a position to apply our main Theorem~\ref{thm:convergence}, from which we obtain the consistency of the $k$-means procedure for diffusion distances.

\begin{prop}\label{prop:consistencydiffusion}
Under the assumptions of Theorem~\ref{thm:spectralconvergenceliterature}, the $k$-means procedure performed on the metric measure space 
\[
\mathcal{X}_n = (\XX_n,\, d_t,\, \mu_n),
\]
is consistent in the sense of Theorem~\ref{thm:convergence}.  
That is, for every $\varepsilon > 0$, there exists $N \in \mathbb{N}$ such that for all $n \ge N$ and for every set of $k$-barycenters$S_n \in \mathcal{S}_{\mathcal X_n}(k) \subset \XX_n$ of $\mathcal{X}_n$, there exists a set of $k$-barycenters ${S} \in \mathcal S_{\mathcal X}(k) \subset \M$ of $\mathcal X= (\M, \delta_t, \mu)$ which is $\varepsilon$–close in Hausdorff distance:
\[
d_H(S_n, {S}) \le \varepsilon.
\]
\end{prop}

\begin{rmq}
Proposition \ref{prop:consistencydiffusion} states the consistency of the centroids calculated using the {$k$-}means procedure under a metric learned from the diffusion distance, by applying Theorem \ref{thm:convergence}. Similarly, Corollary \ref{cor:consistenceclusteringundermGH} also applies and gives the consistency of clusters calculated using the {$k$-}means procedure under a metric learned from the diffusion distance.
\end{rmq}

Since the consistency of the clustering procedure based on diffusion distances follows from the spectral convergence of the graph Laplacian used to define these distances, we conclude this section with a brief discussion of the literature on the convergence of graph Laplacians toward continuous operators.

Recall that the origins of spectral methods for dimensionality reduction trace back to 
\cite{takahashi1966minimal, berard1994embedding}, where heat kernels were employed to embed a manifold into an (infinite-dimensional) Hilbert space. 
In practice, an effective approximation of this embedding is obtained by truncating the expansion to the first $k$ frequencies; that is, to the first $k$ eigenvalues, yielding a finite-dimensional representation of the manifold in $\R^k$.  
Observe that letting $k \to \infty$ in the definition of the diffusion distance \eqref{def:diffusiondistance} recovers precisely the original embedding of Takahashi and Bérard–Besson–Gallot.

These ideas motivate a central principle underlying spectral clustering: ideal clusters should correspond to the connected components of the manifold $\M$. The eigenspace associated with the eigenvalue $0$ of the Laplace--Beltrami operator is precisely the space of functions that are constant on each connected component. An important case is that of the Gaussian kernel defined in~\eqref{def:Gaussiankernel}. In that setting, the limiting operator $U$ in~\eqref{def:limitLaplacian} takes the form
\[
U(f)(x) = f(x) - P_{\sigma^2}(f)(x),
\]
where $P_{\sigma^2}$ denotes the heat semigroup at time $\sigma^2$ generated by the Laplace-Beltrami operator. Consequently, under a suitable choice of the bandwidth parameter $\sigma^2 \to 0$ and an appropriate renormalization of the graph Laplacian $L$, one indeed expects the renormalized operator to converge to the Laplace-Beltrami operator.

Let us mention that another way to define a graph Laplacian is to smooth the data using a Gaussian kernel, rather than proceeding as in \eqref{def:graphLaplace}. For example, a normalized graph Laplacian can be defined from the data $X_1,\ldots,X_n$ as (see e.g.~\cite{gine2006empirical})
\begin{equation}\label{def:weigthedgraphLaplacian}
\Delta_{\varepsilon,n}\, g(p)
:= \frac{1}{n\,\varepsilon^{d+2}}
\sum_{i=1}^n 
\frac{\exp \left(-\frac{|X_i-p|^2}{2\varepsilon}\right)}
{\sum_{j \neq i} \exp \left(-\frac{|X_i-X_j|^2}{2\varepsilon}\right)}
\left( g(X_i) - g(p) \right).
\end{equation}
Such operators have the advantage of being defined directly on $\M$, which makes the analysis of their convergence toward the (weighted) Laplace--Beltrami operator more natural, in contrast with the discrete operators in \eqref{def:graphLaplace}.
The convergence of those operators towards the (weighted) Laplace-Beltrami operator on $\M$, as well as the convergence of the associated eigenvalues and eigenfunctions, has been extensively studied in the literature. Hein, Audibert, and von Luxburg \cite{hein2005graphs} were among the first to establish the strong pointwise consistency of a family of graph Laplacians with data-dependent weights toward a weighted Laplace operator, including the case where the data lie on a submanifold.
Belkin and Niyogi \cite{belkin2006convergence} proved the convergence in probability of all eigenvalues of the (weighted) graph Laplacian toward those of the continuous (weighted) Laplace--Beltrami operator, as well as the convergence in probability of the associated eigenfunctions in the $L^2$ norm.  
Independently and almost simultaneously, Giné and Koltchinskii \cite{gine2006empirical} and Hein \cite{hein2006uniform} showed that, almost surely, the (weighted) graph Laplacian operator converges uniformly for a suitable class of functions toward the (weighted) Laplace--Beltrami operator.  
Subsequently, Singer \cite{singer2006graph} provided a convergence rate for this uniform convergence.  
More recently, Wormell and Reich \cite{wormell2021spectral} provided a refined analysis in the case of distributions supported on a hypertorus, improving upon Singer’s rate and even upgrading the $L^2$ convergence of eigenfunctions obtained by Giné-Koltchinskii and Hein to convergence in $L^\infty$.  
In parallel, García Trillos, Gerlach, Hein, and Šlepčev \cite{garcia2020error} and Calder and García Trillos \cite{calder2022improved} established quantitative rates for the Belkin-Niyogi result concerning the convergence of eigenvalues and the $L^2$ convergence of eigenfunctions. Let us also mention the recent paper \cite{wahl2024kernel} by Wahl, where the study of the graph Laplacian is reformulated through kernel PCA by viewing the heat kernel of the manifold as a reproducing-kernel feature map. Even more recently, Xu and Singer \cite{xu2025manifold} identified conditions ensuring pointwise convergence of the graph Laplacian in the broader framework of general metric spaces, rather than smooth manifolds.

Since in our setting we are interested in the \emph{uniform} convergence of eigenfunctions, the $L^2$-type convergence results commonly found in the literature are not sufficient. Nevertheless, such $L^2$ convergence can often be upgraded to uniform convergence by combining compactness arguments from functional analysis with the perturbation theory of linear operators; see, for instance, \cite{kato2013perturbation}. A detailed discussion of these techniques lies beyond the scope of the present work.

\subsection{Wasserstein {\it k-}barycenters}\label{sec:Wasserstein}

Clustering probability distributions has recently received considerable
attention, see for instance
\cite{papayiannis2021clustering, zhuang2022wasserstein, okano2024wasserstein}.
Most of the literature, however, has focused on computational aspects, while
asymptotic guarantees have remained scarce. To the best of our knowledge, \cite{JaffeKmeans} provides the first published
consistency result for Wasserstein $k$-means with $k>1$, and does so in the
general setting of metric spaces under mild assumptions. Previous works
addressed only the case $k=1$ or the Euclidean distance in $\mathbb{R}^{D}$: the consistency of Wasserstein barycenters was established in
\cite{10.3150/13-BEJ585} for a finite number of measures in $\R^\ell$ and with great generality in \cite{LeGouicLoubes} (infinite number of measures, even uncountable are allowed and the authors consider the Wasserstein space over a general geodesic metric space). Convergence rates were later obtained in
\cite{LeGouicParisRigollet} for the $2$-Wasserstein distance of probability measures over Euclidean space.

In this section, we address the case of the Wasserstein space built over a metric
space whose distance is unknown and must therefore be learned from the data.
This leads to a \emph{learned Wasserstein distance} on the space of probability
measures, with respect to which we compute Wasserstein $k$-means and establish
their consistency.

Before turning to the case of an unknown distance, let us first highlight how
our result naturally generalizes Jaffe’s theorem \cite{JaffeKmeans} in the
setting of a fixed, known metric.

Assume that $(\XX,d)$ is a compact metric space and let  $d_W$ be the $L^2$-Wasserstein distance induced by the quadratic cost $d^2$
on $\XX$. The compactness of $\XX$ guarantees that $(\mathcal P(\XX), d_W)$ is also a  compact metric space.
Interestingly, in this case  $d_W$ convergence is  equivalent to weak convergence. 

Consider a sequence of probabilities  $(\mu_n)_{n\geq 1}$ defined on $\mathcal P(\XX)$, converging weakly towards $\mu$ and define 
\begin{equation}
\label{spaces}
\mathcal{X}_n = \big(\mathcal P(\XX),\, d_W,\, \mu_n\big), \quad \mathcal{X} = \big(\mathcal P (\XX),\, d_W,\, \mu\big).
\end{equation}

Since we have the weak convergence $\mu_n \to \mu$, it is straightforward to check that $\mathcal{X}_n$
converges in the measured Gromov-Hausdorff topology to $\mathcal{X}$.
 
Therefore, Corollary~\ref{cor:embeddedmGH} applies, which yields the consistency of Wasserstein $k$-means.
In particular, this recovers the result of
\cite{10.3150/13-BEJ585, LeGouicLoubes} in the case $k=1$ and Jaffe’s result
\cite{JaffeKmeans} for all $k \ge 1$.

To be more precise, Jaffe did not explicitly state this result in the context of Wasserstein $k$-barycenters, but it follows immediately from an application of his main theorem \cite[Theorem~1]{JaffeKmeans} to this setting.

An important and practically relevant instance of this setting is the following. We are given $m \ge k$
probability measures $\alpha_1,\dots,\alpha_m$ on $\XX$. 
We wish to find a  Wasserstein $k$-barycenter of this measures.  To fit the framework presented in this section,  let $\mu$ be a discrete probability on $\mathcal P(\XX)$, assigning mass $m^{-1}$ to $\alpha_i$, for $i=1, \ldots, m$. Namely, 
\begin{equation}
    \label{mu_pob_wass}
    \mu=\frac{1}{m}\sum_{i=1}^m \delta_{\alpha_i}.
\end{equation}
We look for an element in $\mathcal S_{\mathcal X}(k)$, for ${\mathcal X}$ defined in  \eqref{spaces}, with $\mu$ given in \eqref{mu_pob_wass}.
In
practice, each distribution $\alpha_i$ is only observed through a finite
sample of $n$ points $\XX_n^i\subset\XX$. A natural approach is therefore to replace each
$\alpha_i$ by its empirical measure $$\widehat \alpha_i = \frac{1}{|\XX_n^i|}\sum_{x\in\XX_n^i} \delta_{x},\quad \hbox{for $i=1, \ldots, m$}
$$ and to consider
\[
\mu_n = \frac{1}{m}\sum_{i=1}^m \delta_{\widehat {\alpha_i}}.
\]
the uniform distribution on $\widehat \alpha_1,\dots,\widehat \alpha_m$. Since $\widehat \alpha_i \to
\alpha_i$ weakly as the sample size grows (and in $d_W)$, we obtain the weak convergence of $\mu_n$ to $\mu$. Therefore, the above
consistency result applies, and the empirical Wasserstein $k$-barycenter (or baryncenters) of
 $\widehat \alpha_1,\dots,\widehat \alpha_m$ converges to the Wasserstein $k$-barycenter of $\alpha_1,\ldots, \alpha_m$, in the sense of \eqref{eq:convergence}.

Let us mention that we do not claim any novelty in this recovery of Jaffe’s result, since the proof
of our Theorem~\ref{thm:convergence} relies on \cite[Lemma~4]{JaffeKmeans}, which
establishes the continuity of the clustering functional. However, our
formulation via measured Gromov-Hausdorff convergence both conceptualizes
Jaffe’s result and, more importantly, provides a natural framework for a generalization to the case of unknown metrics, which is the focus of the next paragraph.\\

We now turn to the more challenging and realistic situation where $d$ is not
observed. In this case, we first construct a learned empirical distance $d_n$
from the sample. Using $d_n$, we define an empirical Wasserstein distance
$d_W^{(n)}$ on the space of probability measures over the sample $\XX_n$. We then compute
Wasserstein $k$-barycenters with respect to this learned distance. The goal is
to show that this full procedure, learning a metric and then performing
Wasserstein $k$-means, is still a consistent estimator of the population
$k$-barycenter.
 
Recall that  $(\XX,d)$ is a compact metric space, and $(\mathcal P(\XX),d_W)$ denote its
associated Wasserstein space. Consider a sequence of compact metric spaces
$(\XX_n,d_n)$ converging to $(\XX,d)$ in the Gromov-Hausdorff topology. By
definition, there exist numbers $\varepsilon_n \to 0$ and $\varepsilon_n$-isometries
\[
h_n : (\XX_n,d_n) \longrightarrow (\XX,d).
\]
For each $n$,  let $d_W^{(n)}$  denote the Wasserstein distance defined between probability measured in $(\XX_n, d_n)$ and consider the compact spaces $(\mathcal P(\XX_n), d_W^{(n)})$.

It is a classical result that the Gromov-Hausdorff convergence of compact metric
spaces implies the Gromov-Hausdorff convergence of the associated Wasserstein
spaces, see for instance \cite[Theorem~28.6]{villani2008optimal}. We shall use the
following quantitative version.

\begin{thm}\label{prop:cvlearnedW}
Let $h_n : (\XX_n,d_n) \to (\XX,d)$ be an $\varepsilon_n$-isometry. Then the map
\[
\begin{aligned}
h_n^{\#} : (\mathcal P (\XX_n),d_W^{(n)}) &\longrightarrow (\mathcal P(\XX),d_W), \\
\alpha &\longmapsto h_n\#\alpha,
\end{aligned}
\]
is an $8\bigl(\varepsilon_n + \sqrt{\varepsilon_n\,\mathrm{diam}(\XX)}\bigr)$-isometry. In
particular,
\[
\sup_{\alpha,\beta \in \mathcal P(\XX_n)}
\left|\,d_W^{(n)}(\alpha,\beta) - d_W \left(h_n^{\#}\alpha,\,h_n^{\#}\beta\right)\right|
\le 8\left(\varepsilon_n + \sqrt{\varepsilon_n\,\mathrm{diam}(\XX)}\right).
\]
\end{thm}

Assume in addition that we are given a probability measure $\mu \in \mathcal P (\mathcal P (\XX))$, together with a sequence $(\mu_n)_n$, with $ \mu_n\in \mathcal P (\XX_n)$, such that $h_n^\#\mu_n$ converges weakly to $\mu$. Then, by Theorem~\ref{prop:cvlearnedW}, the metric measure spaces
\[
\bigl(\mathcal P (\XX_n),\, d_W^{(n)},\, \mu_n\bigr)
\]
converge to $\bigl(\mathcal P (\XX),\, d_W,\, \mu\bigr)$ in the measured Gromov-Hausdorff topology. Consequently, Theorem~\ref{thm:convergence} applies and yields the consistency of Wasserstein $k$-barycenters in this learned setting. \\

A particularly relevant situation, is when we consider $A$, $\mathbb X_n^i$, $\widehat  \alpha_i$ and $\mu_n$ as before. Letting $\XX_n = \bigcup_{i=1}^m \XX_n^i$, consider any consistent estimator $d_n$ of the ground metric $d$ (see sections \ref{ssec:fermat} and \ref{sec:spectraldistances}). We may then define $d_W^{(n)}$ as the corresponding learned Wasserstein distance on $\mathcal P (\XX_n)$. Thanks to Theorem~\ref{prop:cvlearnedW}, this implies the following measured Gromov-Hausdorff convergence:
\[
\bigl(\mathcal P (\XX_n),\, d_W^{(n)},\, \mu_n\bigr)
\;\xrightarrow[n\to\infty]{mGH}\;
\bigl(\mathcal P (\XX),\, d_W,\, \mu\bigr).
\]
Note that in this case, the $\varepsilon_n$-isometries are simply the natural embeddings since $\XX_n\subset\XX$.
Therefore, by Theorem~\ref{thm:convergence}, the empirical Wasserstein $k$-barycenter of $\mu_n$ (computed with respect to the learned Wasserstein distance $d_W^{(n)}$) converges to the Wasserstein $k$-barycenter of $\mu$, that is, to the barycenter of the original distributions $\alpha_1,\dots,\alpha_m$. We conclude with the following theorem.

\begin{thm}\label{thm:convergenceforWasserstein}
Let $\alpha_1,\dots,\alpha_m$ be $m\ge k$ probability distributions on a compact metric space $(\XX,d)$, and let
\[
\XX_n := \bigcup_{i=1}^m \XX_n^i,
\]
where, for each $i$, the set $\XX_n^i$ consists of $n$ i.i.d.\ samples drawn from $\alpha_i$. Consider any consistent estimator $d_n$ of the ground metric $d$ (i.e., $(\XX_n,d_n)$ converges in GH to $(\XX,d)$) and let $d_W^{(n)}$ be the $L^2$-Wasserstein distance computed on $(\XX_n,d_n)$; let $\widehat \alpha_i\in \mathcal P (\XX_n)$ be the empirical measure associated with $\XX_n^i$; and let $\mu_n\in \mathcal P (\mathcal P (\XX_n))$ be the uniform distribution on the set $\{\widehat \alpha_1,\dots,\widehat \alpha_m\}$. Then the following holds:

For every $\varepsilon>0$, there exists $N\in\mathbb{N}$ such that for all $n\ge N$, for every set of $k$-barycenters
$S_n \in \mathcal{S}_{\mathcal X_n}(k) \subset \mathcal P (\XX_n)$ of $\mathcal X_n=(\mathcal P (\XX_n), d_W^{(n)}, \mu_n)$, there exists a corresponding set of $k$-barycenters $S \in \mathcal{S}_{\mathcal X}\subset \mathcal P (\XX)$ of $\mathcal X=(\mathcal P (\XX), d_W, \mu)$ that is $\varepsilon$-close in the Hausdorff distance, i.e.
\[
d_H(S_n, {S}) \le \varepsilon.
\]
\end{thm}

\begin{rmq}
Theorem \ref{thm:convergenceforWasserstein} states the consistency of the centroids calculated using the {\it k-}mean procedure under the learned Wasserstein distance, by applying Theorem \ref{thm:convergence}. 
Note that, similarly, Corollary \ref{cor:consistenceclusteringundermGH} also applies and gives the consistency of clusters calculated using the {\it k-}mean procedure under the learned Wasserstein distance.
\end{rmq}

\subsection{First passage percolation}

In this subsection we prove convergence of $k$-barycenters for first passage Percolation (FPP) models. Up to our knowledge, these results are new even for $k=1$.

FPP is a model proposed by Hammersley and Welsh \cite{hammersley1965first} to describe the propagation of a fluid or an infection in a random media. To fix ideas, we consider the $\ell-$dimensional lattice $\mathbb L^\ell=(\ZZ^\ell,\EE^\ell)$ in which $x,y\in \ZZ^\ell$ are declared neighbors (i.e. $\{x,y\}\in \EE^\ell$) if and only if the $1-$norm $|x-y|_1=1$. We denote this by $x\sim y$. To each edge $e\in \EE^\ell$ we assign a random passage time $\tau_e\in\R_+$. The random variables $(\tau_e)_{e\in \EE^\ell}$ are i.i.d. and called the {\em passage times.} For simplicity we assume here that their distribution has a density with respect to the Lebesgue measure, but this is not strictly necessary. For $x,y \in \ZZ^\ell$, a {\em path} from $x$ to $y$ is a sequence of edges $e_1, e_2, \dots, e_m$ in $\EE^\ell$ such that consecutive edges $e_i, e_{i+1}$ share exactly one point. The {\em cost} of such a path is $T(\gamma):=\sum_{e\in \gamma}\tau_e$. The {passage time} from $x$ to $y$ is defined by \(T(x,y) = \inf_\gamma T(\gamma)\), the infimum is over all paths from $x$ to $y$. Since $\PP(\tau_e=0)=0$, the pair $(\ZZ^\ell, T)$ is a (random) metric space. For a given Borel set $D \subset \mathbb R^\ell$, we can consider FPP in $\XX_n= \frac1n(n D\cap \mathbb Z^\ell)$ and $d_n(x,y) = T(nx,ny)/n$, which is easily seen to be a metric space as well. We equip this metric space with $\mu_n$, the uniform measure on $\XX_n$. The translation of the celebrated limit shape theorem  \cite[Theorem 2.34]{auffinger201750} states that $(\XX_n,d_n)$ converges in Gromov-Hausdorff sense to the metric space $(D,d)$, where $d$ is a certain (deterministic) metric that depends on the edge weights for which it is very difficult to obtain precise information. By standard weak convergence arguments, it follows that $\mu_n$ converges weakly towards the Lebesgue measure ${\rm vol}$ on $D$. 

As a consequence, we get that the metric measure spaces $(\XX_n,d_n,\mu_n)$ converge in mGH topology towards $(D,d,{\rm vol})$. It is natural to ask in this situation if the {\em k}-barycenters do converge. 
Since we have no information about $d$, we cannot say much about the {\em k}-barycenter of the limit. However, from the mGH convergence and Theorem~\ref{thm:convergence}, we obtain that the empirical {\em k}-barycenter of the space $(\mathcal{X}_n, d_n, \mu_n)$ is a consistent estimator of it. More precisely, 
\begin{prop}
Almost surely, for all $\varepsilon>0$, there is a $N\in\mathbb{N}$, such that for all $n\geq N$, for every set of { $k$}-barycenters $\mathcal{S}_k\subset\mathcal{X}_n\subset D$ for $(\mathcal{X}_n, d_n, \mu_n)$, there exists a corresponding set of {$k$}-barycenters $\mathcal{S}\subset D$ for $(D, d, \mathrm{vol})$ that is $\varepsilon$-close in the Hausdorff distance, \textit{i.e.}, $d_H(\mathcal{S}_k,\mathcal{S})\leq \varepsilon$.
\end{prop}

\begin{proof}
The proof follows the same steps as in Theorem \ref{thm:fpp}.
\end{proof}

\noindent Alternatively, as it is more common in FPP, we can consider for every time $t$, the ball centered at the origin,
\[
B(t)=\{y \in \ZZ^\ell \colon T(0,y) < t\}.
\]
Again, for $d_t(\frac{x}{t},\frac{y}{t})=t^{-1}T(x,y)$, we have that $\mathcal X_t:=(\frac 1 t B(t), d_t)$ is a metric space that we also equip with the uniform measure $\mu_t$. Note that this metric measure space admits a unique barycenter that we denote $b_t$. The uniqueness comes from the observation that since the random passage times $\tau_e$ admit a density with respect to the Lebesgue measure, we have that for any distinct $x,x'\in \mathcal X_t$, it holds that almost surely
\[
\sum_{y\in B(t)} d_t^2\left(\frac x t, \frac y t \right) \neq \sum_{y\in B(t)} d_t^2\left(\frac {x'} t,\frac y t \right).
\]
The convergence $\mathcal X_t \underset{t\to\infty}{\longrightarrow} \mathcal X$ in GH sense follows from the limit shape theorem \cite[Theorem 3.3]{cox1981some}, where $\mathcal{X} = (\mathcal B_0,d)$ for some deterministic symmetric convex compact set $\mathcal B_0\in\R^\ell$, which induces a norm that depends on the distribution of $\tau_e$, and $d$ is the distance induced by this norm. In this theorem in fact Hausdorff convergence is proved for the shapes in $\R^\ell$, but in our context it is equivalent to Gromov-Hausdorff convergence \cite{Gromov}. See also \cite[Theorem 2.34]{auffinger201750}. In the following theorem we extend the GH convergence to $mGH$ convergence when both spaces are equipped with uniform measure respectively (which are different). 

\begin{thm}\label{thm:fpp}
Let $\{\tau_e\}$ be i.i.d. continuous passage times on the edges of $\mathbb{Z}^\ell$ satisfying $\mathbb{E}e^{\alpha\tau_e}<\infty$ for some $\alpha>0$. Let $\mu_t$ be the uniform probability measure on the finite set $B(t)/t$ and $\mathcal{B}_0\subset\mathbb{R}^\ell$ be the asymptotic shape given by the shape theorem, $\|\cdot\|_\tau$ the associated norm, i.e. $\|x\|_\tau=\lim_{n\to\infty}\frac{1}{n}T(0,nx)$,
and set
\[
d(s,s')=\|s-s'\|_\tau \quad(s,s'\in\mathbb{R}^\ell).
\]
Denote by $\mu$ the uniform probability measure on $\mathcal{B}_0$. Then, almost surely as $t\to\infty$, 
$$(B(t)/t,d_t,\mu_t)\xrightarrow{mGH}(\mathcal{B}_0,d,\mu).
$$
\end{thm}

\begin{corol}
Under the hypothesis of Theorem \ref{thm:fpp}, the Fr\'echet mean of $(B(t)/t,d_t,\mu_t)$, $b_t$, converge to the set of Fr\'echet means of $\mathcal X=(\mathcal{B}_0,d,\mu)$ in the sense that, as $t\to \infty$,
\[
\mathrm{dist}(b_t,\mathcal S_{\mathcal X}(1)) \to 0, \qquad \text{almost surely.} 
\]
\end{corol}
Here \rm{dist} is the distance from a point to a set induced by $d$ as provided by Theorem \ref{thm:convergence}, but recall that $d$ is induced by a norm in $\R^\ell$ and hence this convergence to zero is equivalent to convergence to zero with the Euclidean distance from a point to a set.

\begin{rmq}\label{rem:symmetry}
Since the norm $\|\cdot\|_\tau$ is symmetric $\mathcal{B}_0$ is centrally symmetric
and therefore $0$ belongs to the set of Fr\'echet means
of $(\mathcal{B}_0,d,\mu)$.  Whether this set reduces to the single point $\{0\}$ is equivalent
to the strict convexity of $\mathcal{B}_0$, a property that, although conjectured to be true for continuous passage times, remains open for general
continuous passage time distributions in dimensions $\ell\ge3$ \cite{auffinger201750}.
\end{rmq}

\begin{proof}
The proof is divided in three steps. First we prove the uniform convergence of the metrics, next we show that the immersion $h\colon t^{-1}B(t) \to \R$,  $h(x/t)=x/t$ is an $\varepsilon_t$-isometry, with $\varepsilon_t\to0$ as $t\to \infty$. Finally, we prove the weak convergence of measures $\mu_t\to\mu$. All the statements hold almost surely.

{\em Step 1. Uniform convergence of the metrics.} 
Let $\varepsilon>0$. By the shape theorem, for large $t$ we have $B(t)\subset(1+\varepsilon)t\mathcal{B}_0$,
so there exists $M>0$ such that $\|y-x\|_1\le Mt$ for all $x,y\in B(t)$.

Talagrand's concentration bounds (\cite[Theorem~3.13]{auffinger201750}) combined with Alexander's methods to bound the non-random fluctuations (\cite[Theorem 3.33]{auffinger201750}) provides $C_1,C_2>0$ such that for any $x,y\in\mathbb{Z}^d$ with $\|y-x\|_1$ large enough,
\[
\mathbb{P} \left(\bigl|T(x,y)-d(x,y)\bigr|>{\varepsilon}\|y-x\|_1\right)
\le C_1\exp \bigl(-C_2\|y-x\|_1\bigr).
\]
In particular, for $x , y \in B(t)$,
\[
\mathbb{P} \left(\bigl|T(x, y)-d(x , y )\bigr|>{\varepsilon Mt}\right)
\le C_1\exp \bigl(-C_2Mt\bigr).
\]

Since $|B(t)|=O(t^d)$, the number of pairs $(x,y)$ with $x,y\in B(t)$ is $O(t^{2d})$.
A union bound gives
\[
\mathbb{P}\left(\sup_{x,y\in B(t)}\bigl|T(x,y)-d(x,y)\bigr|>{\varepsilon M}t\right)
\le C't^{2d}\exp\bigl(-C_2Mt\bigr).
\]

Dividing by $t$ and noting that $d(t^{-1}x,t^{-1}y)=t^{-1}d(x,y)$, we obtain
\[
\mathbb{P} \left(\sup_{x/t,\, y/ t\in B(t)/t}\bigl|d_t(x/t,y/t)-d(x/t,y/t)\bigr|>{\varepsilon M}\right)
\le C't^{2d}\exp \bigl(-C_2Mt\bigr).
\]

The right‑hand side is summable over $t$, so by Borel--Cantelli the supremum converges to $0$ almost surely.

{\em Step 2. } The shape theorem gives convergence in Hausdorff distance $d_H(t^{-1}B(t),\mathcal B_0)\to 0$ a.s. as $t\to \infty$. Let $C$ be such that $d(s,s')\le C\|s-s'\|_2$ for every $s,s'\in \R^\ell$. Taking $\varepsilon_t= C^{-1}d_H(t^{-1}B(t),\mathcal B_0)\to 0$, we obtain for every $s\in\mathcal B_0$
\[
\mathrm{dist}(s,h(t^{-1}B(t))) = \min_{x\in B(t)} d(x/t,s) \le  C \min_{x\in B(t)} \| x/t -s \|_2 \le C d_H(t^{-1}B(t),\mathcal B_0) = \varepsilon_t,
\]
proving that $h$ is an $\varepsilon_t$-isometry.

{\em Step 3. Weak convergence of uniform measures.}
We have $h_\#\mu_t=\mu_t$ (viewed as a discrete measure on $\R^\ell$).
To see that $\mu_t\to\mu=h_\#\mu$ weakly on $\mathcal{B}_0$, let $A\subset\mathcal{B}_0$ be closed.
For any $\delta>0$ set $A_\delta=\{s\in\mathcal{B}_0:\operatorname{dist}(s,A)\le\delta\}$;
then $A\subset A_\delta$ and $\operatorname{Vol}(A_\delta)\to\operatorname{Vol}(A)$ as $\delta\to0$.
Using the inclusions $(1-\delta)t\mathcal{B}_0\subset B(t)\subset(1+\delta)t\mathcal{B}_0$
provided by the shape theorem,
\[
\mu_t(A)=\frac{|B(t)\cap tA|}{|B(t)|}
\le\frac{|tA_\delta\cap\mathbb{Z}^d|}{|(1-\delta)t\mathcal{B}_0\cap\mathbb{Z}^d|}
\stackrel{t\to\infty}{\longrightarrow}\frac{\operatorname{Vol}(A_\delta)}{(1-\delta)^d\operatorname{Vol}(\mathcal{B}_0)}.
\]
Letting $\delta\to0$ gives $\limsup_{t\to\infty}\mu_t(A)\le\mu(A)$.  By the Portmanteau theorem,
$\mu_t\to\mu$ weakly. This finishes the proof.

\end{proof}

\subsection{Discrete approximations of length spaces}

It is a well-known fact that every compact length space can be obtained as a Gromov–Hausdorff limit of finite graphs, see for instance \cite[Proposition~7.5.5]{Burago}. Therefore, it is natural to look for such approximations that preserve the $k$-barycenter, \textit{i.e.}, such that the $k$-barycenter with respect to the (weighted) empirical measure on the approximating graph converges to the $k$-barycenter of the approximated length space equipped with some measure. 
Note that this is not trivial, since different approximating sequences $(\mathbb{X}_n, d_n)$ converging in the Gromov–Hausdorff sense to the same metric space $(\mathbb{X}, d)$ may, when equipped with the uniform measure $\mu_n$ on $\mathbb{X}_n$, converge to different metric measure spaces. 
The goal of this section is to highlight that our Theorem~\ref{thm:convergence} provides a positive answer to this question, as it shows that $k$-barycenters are indeed preserved whenever the approximating graph is not only an approximation in the GH topology but also in the mGH topology.

\begin{prop}
All compact length metric measure spaces $(\mathbb{X}, d, \mu)$ can be approximated in the measured Gromov–Hausdorff topology by finite graphs $(G_n, d_n, \mu_n)$ equipped with a distance $d_n$ and a measure $\mu_n$. In particular, for every $\varepsilon > 0$, if $n$ is large enough, any set of $k$-barycenters $S_n \in \mathcal{S}_{\mathcal X_n}(k)$ of  $\mathcal X_n=(G_n, d_n, \mu_n)$ is $\varepsilon$-close, in the Hausdorff distance, to some set $S \in \mathcal{S}_{\mathcal X}(k)$ of $k$-barycenters of $\mathcal X=(\mathbb{X}, d, \mu)$.
\end{prop}

\begin{proof}
See the discussion below.
\end{proof}

We have seen in Subsection~\ref{ssec:iid} that for any fully supported measure $\mu$ on $\mathbb{X}$, it is possible to obtain a metric measure space $(\mathbb{X}_n, d_n, \mu_n)$ that converges in the mGH sense towards $(\mathbb{X}, d, \mu)$ by choosing $\mathbb{X}_n$ to be an i.i.d.\ sample drawn from $\mu$, and $d_n$ the distance induced by $d$. 
However, when some information about the space $(\mathbb{X}, d)$ is available, it is preferable to choose specific (non-random) discretizations of the space, for instance, to approximate a torus with $n$ points. A standard procedure to discretize a measure space in an optimal way is through the \emph{quantization} of the measure~\cite{Graf}. An $n$th-order quantization $q_n$ of a probability measure $\mu$ is a probability measure supported on at most $n$ points that solves
\begin{equation}\label{eq:quantization}
    q_n \in \underset{q \in \mathcal{P}_n}{\arg\min} \, d_W(\mu, q),
\end{equation}
where $d_W$ denotes, the $p-$th order Wasserstein distance and $\mathcal{P}_n$ is the set of probability measures on $\mathbb{X}$ supported on at most $n$ points.
Note that this problem is equivalent to the $k-$means problem with $k=n$. However, in the $k-$means problem, $k$ is fixed and corresponds to the number of clusters, whereas in the quantization problem, $k$ can vary, and in particular, the larger it is, the finer the quantization will be relative to the original measurement.
It follows that if $\mu$ is fully supported, any sequence $(q_n)_{n \ge 1}$ of minimizers in~\eqref{eq:quantization} satisfies $\sup_{x \in \mathbb{X}} d(x, \mathrm{supp}\, q_n) \to 0$. This is because the minimizers do better than $n$ i.i.d. points, for which the limit holds.
As a consequence, such quantized measures can be used to construct approximations of a metric measure space in the mGH topology. More generally, let $\varepsilon_n \to 0$, and let $\mathbb{X}_n$ be a $\delta_n$-net of $(\mathbb{X}, d)$, for some $\delta_n$ to be fixed later. We construct the $\varepsilon_n$-neighborhood graph $G_n$, \textit{i.e.}, the graph with vertex set $\mathbb{X}_n$, where two points $x, y \in \mathbb{X}_n$ are connected if and only if $d(x, y) < \varepsilon_n$. 
The length of each edge $\{x, y\}$ is set to be $d(x, y)$, and we denote by $d_n$ the induced graph distance in $G_n$. If $\delta_n < \frac{1}{4} \varepsilon_n^2 / \mathrm{diam}(\mathbb{X})$, then by~\cite[Proposition~7.5.5]{Burago}, it holds that 
\[
\sup_{x, y \in \mathbb{X}} |d_n(x, y) - d(x, y)| \le \varepsilon_n.
\]
Proceeding as in Subsection~\ref{ssec:iid}, we obtain that $(\mathbb{X}_n, d_n, \mu_n) \to (\mathbb{X}, d, \mu)$ in the mGH sense if and only if $\mu_n \to \mu$ weakly. 
As mentioned above, this can be achieved in several different ways, for instance, by taking $\mu_n$ to be the uniform measure on $\mathbb{X}_n$ when $\mathbb{X}_n$ is chosen as an i.i.d.\ sample from $\mu$. One can also take $\mu_n$ to be the quantized measure $q_n$, and $\mathbb{X}_n$ to be the $n$-centers of this quantization $q_n$ of $\mu$. Note that in this latter case, where $\mathbb{X}_n$ is chosen as the $n$-centers of a quantization of $\mu$, one can instead take $\mu_n$ to be the uniform measure on these $n$-centers rather than the quantized measure $q_n$. 
In this situation, the sequence $(\mu_n)_{n \ge 1}$ does admit a weak limit, however, this limit is not equal to $\mu$. We refer the reader to~\cite[Sections~6 and~7]{Graf} for the Euclidean case, and to~\cite[Theorem~1.2]{Kloeckner} for absolutely continuous measures with respect to the volume measure ${\rm vol}_g$ of a Riemannian manifold $(\M, g)$.
To summarize, in both cases, if $\mu$ admits a density $\rho$ with respect to the Lebesgue measure (resp.\ the volume measure), then the sequence $\bar{\mu}_n$ of probability measures that are uniform on the support of $q_n$ converges weakly to the normalized multiple of $\rho^{-\ell/(p+\ell)}$, where $\ell$ denotes the dimension of the ambient space and $p$ is the Wasserstein parameter fixed for the quantization procedure~\eqref{eq:quantization}. Therefore, if we want to obtain the barycenter of $\mu$, we need to consider quantizations of the measure with density proportional to $\rho^{-(p+\ell)/\ell}$ in order to compensate.
In each of these cases, Theorem~\ref{thm:convergence} applies, and we can summarize it in the following proposition.
\begin{prop}
Let $k\in\mathbb{N}^*$, $p \ge 1$, and let $(\mathbb{X}, d, \mu)$ be a compact length metric measure space. Then the $k$-means estimator is consistent when computed with respect to the following discretizations:
\begin{itemize}
    \item $(\mathbb{X}_n, d_n, \mu_n)$, where $\mathbb{X}_n$ is an i.i.d.\ sample drawn from $\mu$, $d_n$ is the distance on $\mathbb{X}_n$ induced by $d$, and $\mu_n$ is the uniform measure on $\mathbb{X}_n$;
    \item $(\mathbb{X}_n, d_n, q_n)$, where $q_n$ is an $n$th-order quantization of $\mu$ in the sense of \eqref{eq:quantization}, $\mathbb{X}_n$ is the set of $n$-centers of the quantization (i.e., the support of $q_n$), and $d_n$ is the distance on $\mathbb{X}_n$ induced by $d$.
\end{itemize}
Moreover, if $\mathbb{X}$ is a compact Riemannian manifold of dimension $\ell$, and $\mu$ admits a density $\rho$ with respect to the volume measure, then the $k$-means estimator is also consistent when computed with respect to the following discretization:
\begin{itemize}
    \item $(\mathbb{X}_n, d_n, \bar{\mu}_n)$, where $\mathbb{X}_n$ is the set of $n$-centers of a quantization $q_n$ of the probability measure $\nu \propto \rho^{-(p+\ell)/\ell}\, d\mathrm{vol}$, $d_n$ is the distance on $\mathbb{X}_n$ induced by $d$, and $\bar{\mu}_n$ is the uniform measure on $\mathbb{X}_n$.
\end{itemize}
\end{prop}

Recall that by consistent, we always mean that for every $\varepsilon > 0$, if $n$ is large enough, any set of $k$-barycenters $S_n$ for $(\XX_n, d_n, \mu_n)$ is $\varepsilon$-close, in the Hausdorff distance, to some set ${S}$ of $k$-barycenters for $(\mathbb{X}, d, \mu)$.

\bibliographystyle{abbrv}
\bibliography{biblio}

\end{document}